\documentclass[12pt]{amsart}
\usepackage{amsfonts,latexsym,amsmath, amssymb}
\usepackage{mathrsfs,MnSymbol}
\usepackage{url,color}
\usepackage{upgreek}
\usepackage{fancyhdr}
\usepackage{hyperref}
\usepackage{times}
\usepackage{scalefnt}
\usepackage{url}
\usepackage{cite}
\usepackage{amsmath}

\newcommand\myuline{\bgroup\markoverwith
	{\textcolor{gray}{\rule[0.4ex]{2pt}{3pt}}}\ULon}

\newcommand{\bea}{\begin{eqnarray}}
	\newcommand{\eea}{\end{eqnarray}}
\def\beaa{\begin{eqnarray*}}
	\def\eeaa{\end{eqnarray*}}
\def\ba{\begin{array}}
	\def\ea{\end{array}}
\def\be#1{\begin{equation} \label{#1}}
	\def\eeq{\end{equation}}

\def\lV{\lVert}
\def\rV{\rVert}

\def\div{\mathrm{div}}
\def\d{\partial}
\newcommand{\tr}{\operatorname{tr}}

\newcommand{\Lf}{L^{\infty}}

\def\be{{\beta}}
\def\ga{\gamma}
\def\Ga{\Gamma}
\def\de{\delta}
\def\De{\Delta}
\def\ep{\epsilon}

\def\la{\lambda}

\def\si{\sigma}

\def\om{\omega}
\def\Om{\Omega}
\def\th{\theta}

\def\nab{\nabla}

\def\al{\alpha}
\def\les{\lesssim}

\def\AA{{\mathcal A}}

\def\BB{{\mathcal B}}

\def\EE{{\mathcal E}}

\def\KK{{\mathcal K}}

\def\DD{{\mathcal D}}

\def\AA{{\mathcal A}}

\def\XX{{\mathcal X}}

\def\R{{\mathbb{R}}}

\def\S{{\mathbb{S}}}

\def\Z{{\mathbb{Z}}}

\def\P{{\mathbb P}}

\newcommand{\<}{  \langle   }
\renewcommand{\>}{  \rangle   }

\newtheorem{theorem}{Theorem}[section]
\newtheorem{lemma}[theorem]{Lemma}
\newtheorem{proposition}[theorem]{Proposition}

\newtheorem{definition}[theorem]{Definition}
\newtheorem{remark}[theorem]{Remark}

\numberwithin{equation}{section}

\usepackage{xcolor}

\topmargin - .23 in

\begin{document}
	
	%\scalefont{1.10}
	
\title[Almost global well-posedness for hyperbolic liquid crystal]
{Almost global well-posedness of 2-D Ericksen-Leslie's hyperbolic liquid crystal model for small data}
\author[J. Huang]{Jiaxi Huang$^1$}

\author[N. Jiang]{Ning Jiang$^{2}$}

\author[L. Zhao]{Lifeng Zhao$^3$}

\address{$^1$ School of Mathematics and Statistics, Beijing Institute of Technology, 
	\newline\indent
	Beijing
	100081, P.R. China}
\email{\href{mailto:jiaxih@bit.edu.cn}{jiaxih@bit.edu.cn}}

\address{$^2$School of Mathemtical and Statistics, Wuhan University, 
	\newline\indent
	Wuhan 430072, P. R. China}
\email{\href{mailto:njiang@whu.edu.cn}{njiang@whu.edu.cn}}

\address{$^3$School of Mathemtical Sciences, University of Science and Technology of China, 
	\newline\indent
	Hefei 230026, P. R. China}
\email{\href{mailto:zhaolf@ustc.edu.cn}{zhaolf@ustc.edu.cn}}

%\subjclass[2010]{35M30, 35L52,76D03, 35B35, 35B65}


%\thanks{Acknowledgments.}
\begin{abstract}
This article is devoted to the two dimensional simplified Ericksen-Leslie's hyperbolic system for incompressible liquid crystal model, where the direction $d$ of liquid crystal molecules satisfies a wave map equation with an acoustical metric. 
We established the almost global well-posedness for small and smooth initial data near the constant equilibrium. Our proof relies on the idea of vector-field method and ghost weight method.
There are two key ingredients in our proof: (i) Inspired by the gauge theory in Tataru \cite{Tataru,Tataru05}, we reformulate the wave map equation into a free wave equation with acoustical metric, where the nonlinearity is annihilated due to the geometry of $\S^1$; (ii) Motivated by the ghost weight method in Alinhac \cite{A01}, we introduce a new and important ``good unknown", the velocity $u$, which provides the additional dissipation $u/\<t-r\>\in L^2_tL^2_x$.
These new observations turn out to be extremely 
crucial in resolving the system in low dimensions.  
\end{abstract}

\date{\today}
\maketitle

\setcounter{tocdepth}{2}
%\pagenumbering{roman} \tableofcontents \newpage \pagenumbering{arabic}

\section{Introduction}
Hyperbolic liquid crystal system was introduced by Ericksen  and Leslie in their pioneering works \cite{Ericksen-1961-TSR} and \cite{Leslie-1968-ARMA} respectively during the early days of mathematical hydrodynamic theory in the 1960's, see also their other related works \cite{Ericksen-1961-TSR, Ericksen-1987-RM, Ericksen-1990-ARMA, Leslie-1968-ARMA, Leslie-1979} and the survey article by Lin and Liu \cite{Lin-Liu-2001}. The general Ericksen-Leslie's system in dimensions $n=2,\ 3$ consists of the following equations of the velocity field $u(x,t)\in \mathbb{R}^n$ and the orientation field $d(x,t)\in \mathbb{S}^{n-1}$, and $(x,t)\in \R^n\times \R^+$:
\begin{equation}\label{PHLC}
\begin{aligned}
\left\{ \begin{array}{c}
\partial_t u + u \cdot \nabla u  + \nabla p = - \div (\nabla d \odot \nabla d) + \div \sigma\,, \\
\div u = 0\, ,\\
\rho_1 \ddot{d} = \Delta d + \Gamma d + \lambda_1 (\dot{d} + B d) + \lambda_2 A d\,.
\end{array}\right.
\end{aligned}	
\end{equation}
For the detailed derivation of \eqref{PHLC} from the original form of Ericksen-Leslie's formulation, see, for example \cite{Jiang-Luo-2018}.

In the above system, $\rho_1 > 0$ is the inertial constant, and the superposed dot denotes the material derivative $\partial_t + u \cdot\nabla$, and
\begin{equation*}
\begin{aligned}
A = \tfrac{1}{2}(\nabla u + \nabla^T u)\,,\quad B= \tfrac{1}{2}(\nabla u - \nabla^T u)\,,\end{aligned}
\end{equation*}
represent the rate of strain tensor and skew-symmetric part of the strain rate, respectively. We also define $N = \dot d + B d$ as the rigid rotation part of director changing rate by fluid vorticity. Here $A_{ij} = \tfrac{1}{2} (\partial_j u_i + \partial_i u_j)$, $ B_{ij} = \tfrac{1}{2} (\partial_j u_i - \partial_i u_j) $, $(B d)_i =B_{ki} d_k$, and $(\nabla d \odot \nabla d)_{ij} = \partial_i d_k \partial_j d_k$. The stress tensor $\sigma$ has the following form:
\begin{equation}\label{Extra-Sress-sigma}
\begin{aligned}
\sigma_{ji}=  \nu_1 d_k A_{kp}d_p  d_i d_j + \nu_2  d_j N_i  + \nu_3 d_i N_j  + \nu_4 A_{ij} + \nu_5 A_{ik}d_k d_j   + \nu_6 d_i A_{jk}d_k \,.
\end{aligned}
\end{equation}
These coefficients $\nu_i (1 \leq i \leq 6)$ which may depend on material and temperature, are usually called Leslie coefficients, and are related to certain local correlations in the fluid. Usually, the following relations are frequently introduced in the literatures \cite{Ericksen-1961-TSR,Leslie-1968-ARMA, Wu-Xu-Liu-ARMA2013}.
\begin{equation*}%\label{Coefficients-Relations}
\lambda_1=\nu_2-\nu_3\,, \quad\lambda_2 = \nu_5-\nu_6\,,\quad \nu_2+\nu_3 = \nu_6-\nu_5\,.
\end{equation*}
The first two relations are necessary conditions in order to satisfy the equation of motion identically, while the third relation is called {\em Parodi's relation}, which is derived from Onsager reciprocal relations expressing the equality of certain relations between flows and forces in thermodynamic systems out of equilibrium. Under Parodi's relation, we see that the dynamics of an incompressible nematic liquid crystal flow involve five independent Leslie coefficients in \eqref{Extra-Sress-sigma}. Furthermore, in \eqref{PHLC}, the Lagrangian multiplier $\Gamma$ is (which ensures the geometric constraint $|d|=1$):
\begin{equation*}%\label{Lagrange-Multiplier}
\Gamma = - \rho_1 |\dot{d}|^2 + |\nabla d|^2 - \lambda_2 d^T A d\, .
\end{equation*}

Compared with the extensively studied  parabolic counterpart (see \cite{LW-survey}), the hyperbolic liquid crystal system \eqref{PHLC} has a prominent feature that the direction $d$ satisfies a wave map type equation. When $u=0$, it becomes the classical wave map,  which has been extensively studied in the last decades, here we recall the non-exhaustive lists of results. 
Local well-posedness in $H^s\times H^{s-1}$ was established for all subcritical regularities $s>n/2$ by Klainerman-Machedon for $n\geq 3$ and Klainerman-Selberg when $n=2$ \cite{KM,KSe-1997}.
The much more delicate small data critical problem was started by Tataru \cite{Tataru} in the critical Besov space, and then
completed by Tao \cite{Tao1,Tao2} for wave maps from $\R^{1+n}$ to $\S^m$ in the critical Sobolev space. The dynamic behavior for wave maps from $\R^{1+2}$  to $\S^2$ was obtained by  Sterbenz-Tataru \cite{ST1,ST2} below the threshold. For 1-equivariant wave maps slightly above the threshold, Krieger-Schlag-Tataru \cite{KST-2008} contructed finite time blowup solutions with continuum blowup rate while stable smooth blowup solutions were contructed by Raphael-Rodnianski \cite{RR}.  Recently Duychaertz-Kenig-Martel-Merle \cite{DKMM} and Jrendrej-Lawrie \cite{JL} established the soliton resolution for $k$-equivariant ($k\geq1$) wave maps. For more complete review of the works for the wave map equations, please see the references therein above and \cite{KTV,GG-2017}.

However, the wave map equation satisfied by the direction $d$ becomes quasilinear due to the presence of acoustical metric. In fact, the $d$-equation in \eqref{PHLC} can be viewed as a wave map on a curved manifold $(\R^n,g)$ with the metric  $g=-dt\otimes dt+\sum_{i=1}^n (dx^i-u^idt)\otimes (dx^i-u^idt)$ depending on the velocity $u$ of fluid. Thus the $d$-equation is essentially a quasilinear wave map equation, which brings many  difficulties in studying the behaviors of hyperbolic liquid crystal \eqref{PHLC}. Similar phenomenon also appears in compressible Euler equations, see \cite{LS-2018,DIT,IT-2020}. Hence, one of the key issues is to exploit the structure of nontrivial metric, especially the second-order material derivative.

The study of  hyperbolic liquid crystals is still in its infancy. Recently, Jiang and Luo established in \cite{Jiang-Luo-2018} the well-posedness in the context of classical solutions for the hyperbolic case, i.e. $\rho_1 >0$. More precisely, in \cite{Jiang-Luo-2018} under some natural constraints on the Leslie coefficients which ensure the basic energy law is dissipative, they proved the local-in-time existence and uniqueness of the classical solution to the system \eqref{PHLC} with finite initial energy. Furthermore, with an additional assumption on the coefficients which provides a damping effect, i.e. $\lambda_1 < 0$, and the smallness of the initial energy, the unique global classical solution was established. Here we remark that the assumption $\lambda_1 < 0$ plays a crucial role in the global-in-time well-posedness. Cai-Wang \cite{CW} made progress for the simplied Ericksen-Leslie system, namely, the case with  $\nu_i = 0, i=1\,,\cdots\,,6$, $i\neq 4$. They proved the global regularity of \eqref{PHLC} near the constant equilibrium by employing the vector field method. Later, the general incompressible Ericksen-Leslies system without kinematic transport was considered by Huang-Jiang-Luo-Zhao in \cite{HJLZ1}. They proved the global regularity of \eqref{PHLC} near the equilibrium by space-time resonance method. This strategy was also used to study the simplified compressible Ericksen-Leslie system, which admits more ``wave" features. Huang-Jiang-Luo-Zhao \cite{HJLZ2} showed that the solution of simplified compressible Ericksen-Leslie’s system near the constant equilibrium is
globally well-posed and scattering. However, in two dimensions $n=2$, the global regularity and scattering of hyperbolic liquid crystal for small data is still open.

Finally, we mention some results on other liquid crystal models of hyperbolic type. In \eqref{PHLC}, taking $\sigma=0$ and $u=0$,the system \eqref{PHLC} can be reduced to a so-called nonlinear variational wave equation in dimension one. Zhang-Zheng  studied systematically the dissipative solution and the energy conservative solutions \cite{Zhang-Zheng-CPDE2001,Zhang-Zheng-ARMA2010}. The inertial Qian-Sheng model of liquid crystals is another interesting model which was investigated by De Anna and Zarnescu \cite{DeAnna-Zarnescu-2016}. They derived the energy law and proved the local well-posdedness for bounded initial data and global well-posedness for small data when the coefficients satisfy some further damping property. For the inviscid version of the Qian-Sheng model, Feireisl-Rocca-Schimperna-Zarnescu \cite{FRSZ-2016} proved the global existence of the {\em dissipative solution}. 

\

In the current paper, we consider the \emph{two dimensional} simplified Ericksen-Leslie's parabolic-hyperbolic liquid crystal model in the following form: 
\begin{equation}        \label{ori_sys}
\left\{
\begin{aligned}
&\partial_t u+u\cdot\nabla u+\nabla p  =\Delta u-\div(\nabla d \odot \nabla d) \,, \\
&\div u  =0 \,,\\
&\ddot{d}-\Delta d  =(-|\dot{d}|^2+|\nabla d|^2)d\,,
\end{aligned}
\right.
\end{equation}
on $\R^2\times \R^+$ with the constraint $d\in \S^1$. Here we should note that the direction $d$ of liquid crystal molecules evolves along the wave map flow under the Lorentzian metric $g$ corresponding to the propagation of
sound waves. Generally, the Lorentzian metric $g$ is called \emph{acoustical metric}, which is defined as below:
\begin{definition}[The acoustical metric and its inverse] 
We define the acoustical metric $g$ and the inverse acoustical metric $g^{-1}$ relative to the Cartesian coordinates as follows:
\begin{equation} \label{acMet}
    g:=-dt\otimes dt+\sum_{j=1}^n (dx_j-u_j dt)\otimes (dx_j-u_j dt)\,,
\end{equation}
\begin{equation*}
    g^{-1}:=-(\d_t+\sum_{j=1}^n u_j\d_j)\otimes (\d_t+\sum_{j=1}^n u_j\d_j)+\sum_{j=1}^n\d_j \otimes \d_j \,,
\end{equation*}
where $n$ is the spatial dimensions.
\end{definition}

Now we aim to prove the small data \emph{almost global well-posedness} for simplified Ericksen-Leslie's liquid crystal model near the constant equilibrium $(u,d)= (\vec{0},\vec{i})$ in low dimensions $n=2$. Namely, the solution of simplified (EL) equations with small initial data $\les \ep$ exists on the time interval $[0,e^{C_T/\ep}]$. Different from the case in three dimensions, the worse decays bring many issues in low dimensions. Here we have to pay more attention to investigating the structure of acoustical metric, especially the second-order material derivative, which resolves these difficulties.

Inspired by the gauge theory in Tataru \cite{Tataru,Tataru05}, we formulate the liquid crystal model as a perfect form first. Precisely, we represent $d\in \S^1$ in terms of the angle between unit vector $d$ and $x$-axis $\phi$, namely
\begin{equation*}     %     \label{d_pres}
	d=(\cos\phi,\ \sin\phi)\,.
\end{equation*}
Since the orientation $d$ is considered near a fixed unit vector, then without loss of generality, the angles $\phi$ can be chose near $0$. Then the system (\ref{ori_sys}) can be rewritten as the $(u,\phi)$-system:
\begin{equation}                   \label{Main_Sys}
	\left\{\begin{aligned}
		&\partial_t u-\De u    =-u\cdot\nabla u-\nab p-\div(\nab \phi\odot \nab \phi)\,, \\
		&\div u=0\,,\\
		&(\d_t+u\cdot\nab)^2 \phi-\Delta\phi=0\,.
	\end{aligned}
	\right.
\end{equation}
Here $\div(\nab \phi \odot \nab\phi)=\sum_{j=1}^2 \d_j(\nab\phi \d_j\phi)$. Note that the angle $\phi$ satisfies a free wave equation under the acoustical metric \eqref{acMet}. The above reformulation is the key step towards characterizing the long-time behaviors of the hyperbolic liquid crystal model in low dimensions.
The detailed derivation of \eqref{Main_Sys} will be given in Section 2.

Now we are in the position to state our main result. 
\begin{theorem}           \label{Main_thm}
	Let $N\geq 10$ and $(u_0,\phi_0,\phi_1)\in H^N_\Lambda$. Assume that the initial data $(u_0,\phi_0,\phi_1)$ is small,
	\begin{equation}\label{MainAss_dini}
		\|(u_0,\phi_0,\phi_1) \|_{H^N_\Lambda}\leq \ep \,,
	\end{equation}
	then there exists a unique solution $(u,\phi)\in H^{N}_\Ga$  of Ericksen-Leslie's system \eqref{Main_Sys} on the time interval $[0,e^{C_T/\ep}]$ which satisfies the following bounds for all $t\in[0,e^{C_T/\ep}]$: 
	
	(i) Energy bounds:
	\begin{equation*}%     \label{energybd}
		E^{1/2}_N(t)\les \ep\,.
	\end{equation*}
    
    (ii) Decay estimates:
    \begin{align}      \label{decay-thm}
    	\|u(t)\|_{\Lf}\les \ep\<t\>^{-1/2}\,,\quad \|\nab_{t,x}\phi(t)\|_{\Lf}\les \ep\<t\>^{-1/2}\,.
    \end{align}
\end{theorem}

\begin{remark}
The functional spaces $H^N_\Lambda$, $H^N_{\Ga}$ and energy bounds $E_N$ are defined by using the vector fields associated with the symmetries of \eqref{Main_Sys}, which will be provided in Section \ref{sec-pre}. Moreover, we will utilize the modified energy and ghost weight energy to derive our main energy estimates, where the modified energy is equivalent with $E_N$ as above.
\end{remark}

\begin{remark}
We should mention that the optimal decay estimate of $u$ in \eqref{decay-thm} does not depend on the time interval. This will be beneficial to prove global regularity in further works. 
	Moreover, the decay estimate is sharp in the sense that the decay rate is the same as the linear heat equation if the initial data lies in the energy space. 
\end{remark}

In two dimensions, there are two obstacles that prevent us from proving the above result. 
Firstly, the optimal pointwise decays of velocity $u$ and angle $\phi$ are the poor form $t^{-1/2}$, which shows a big gap compared with the “almost” integrable $t^{-1}$ decay. Secondly, the structure of acoustical metric (or second-order material derivatives) is still unclear, which has a full (codimension 1) set of time resonances. The slow decays and the presence of large sets of resonances require new ideas to control the growth of the solution over time.

Our proof of Theorem \ref{Main_thm} is based on the energy estimates and decay estimates. Some key observations and novelties  are emphasized as follows.
	
\medskip

\emph{1. The geometry of  $\S^1$ annihilates the nonlinearity.}
The wave map equation is usually reformulated in a suitable gauge. Under the gauge, the wave map with general target manifolds will lead to a  \emph{cubic} wave equations as shown by Tataru \cite{Tataru,Tataru05} and Tao \cite{Tao2,TaoRe}. However, the direction field $d$ for the two-dimensional liquid crystals lies in a  circle $\S^1$. Thus we'd like to write the system as a more beautiful formula using the geometry structure of $\S^1$. 
Since the tangent bundle of $\S^1$   is a  line bundle, there does not admit any additional gauge freedom compared with the general target manifold. As a result, the nonlinearity will be annihilated under our gauge choice. 
In this article, there is an alternative natural formulation for two-dimensional liquid crystals in terms of  $\phi$, which represent the angle between $d$ and $x$-axis due to the special geometry of $\S^1$. 
Actually, the angle variable $\phi$ is also widely used in many physical literature related to liquid crystals, see \cite{Leslie-1979}. Hence, the wave map equation \eqref{ori_sys} is essentially a free wave equation with a highly nontrivial metric, the acoustical metric.

\medskip 
\emph{2. Optimal decay estimates of velocity $u$.} 
The optimal decay estimates of velocity $u$ are necessary for our proof. Different from the three-dimensional case, the weaker decay of heat flow $e^{t\De}$ in $\R^2$ makes it more difficult to obtain the decay estimates for $u$. 
Fortunately, the following estimate hold
\[   \|t^{1/2}u\|_{\Lf\Lf}\les \ep+ \|t^{1/2}u\|_{\Lf\Lf}^2 \, .  \]
under the bootstrap assumptions. Then the decay estimate follows from the continuity method. Indeed, the decay rate we obtained is optimal and no additional assumptions in the length of the time interval are required in the decay estimates.

\medskip 
\emph{3. Ghost weight energy and new good unknown.} 
The ghost weight method is originally introduced by Alinhac in \cite{A01} to study the quasilinear wave equations, which is also robust in the parabolic-hyperbolic liquid crystal model.
Precisely, we define the energy functional and weight energy using the ghost weight $e^{-q}$.
Then we deduce two good unknowns, ``$\om\d_t\phi+\nab\phi$" and velocity ``$u$", in the system \eqref{Main_Sys}, where the latter one is first discovered in this article. That is to say, we have the following additional dissipative properties:
\begin{align*}
    \frac{\om\d_t \phi+\nab\phi}{\<t-r\>}\in L^2L^2\,,\quad \frac{ u}{\<t-r\>}\in L^2L^2\,. 
\end{align*}
The above two functions are used to define the so-called ``ghost weight energy" $\mathcal D_\al$.

\medskip 
\emph{4. Acoustical metric (or second-order material derivative).}
The equation of $\phi$ in \eqref{Main_Sys} evolves along a free wave equation under the acoustical metric, where the highly nontrivial metric (or the second-order material derivative) brings many issues in the energy estimates and weighted $L^2$ estimates. Here we have to investigate the worst term that comes from acoustical metric, that is the quadratic term $u\cdot\nab\d_t\phi$.  Given the above good unknowns, we know that $u\cdot\nab \d_t\phi$ admits \emph{null structure}, which is sufficiently in three dimensions. However, in the current paper, we must know more about structure. From the decomposition of $\nab=\om\d_r+\frac{\om^\perp}{r}\d_\th$, the quadratic term can be decomposed as
\begin{align*}
    u\cdot\nab \d_t\phi =u\cdot \om \d_r\d_t\phi+u\cdot \frac{\om^\perp}{r}\d_\th \d_t\phi\,.
\end{align*}
The first term on the right hand side admit ``almost" integrable decay $t^{-1}$ due to $u\cdot \om \les \ep \<t\>^{-1}$. The second term also has the same decay $t^{-1}$ near light cone $r\approx t$, which is the worst area that the nonlinearities are controlled. Hence, with our rigorous analysis, we can obtain the log-growth for the energy of \eqref{Main_Sys}.

\bigskip 
This paper is organized as follows. In Section 2, we will first derive the system \eqref{ori_sys} with a general targetw manifold by energetic variational approach. Then we reformulate the system in two different methods. One of the formulations that introduce variable $\phi$ will be used in the present article. In Section \ref{sec-pre} we carry out a preliminary step in the proof of our main result, namely, we set up the notations for function spaces and state the bootstrap proposition. From this Proposition \ref{Main_Prop}, we can use the continuity method to prove our main Theorem \ref{Main_thm}. In the remaining sections, we turn our attention to the proof of Proposition \ref{Main_Prop}, which contains three parts. In Section \ref{sec-decay}, we give the various decay estimates of $u$ and $\phi$. Particularly, the decays of $\nab u$ in $L^2$, $\nab u,\ \nab^2 u$ and $\d_t u$ in $\Lf$ are also provided. These estimates are also needed in the other two parts. The next Section \ref{sec-L2} is devoted to the estimates of $L^2$ weighted norms $\XX_{N-2}$ and $\XX_N$. Here we show an appropriate bound for the higher order $L^2$ weighted norms. Finally, in Section \ref{sec-Energy}, we use the ghost weight method to prove the modified energy estimates. Here, energy corrections are also needed to deal with derivative loss. In the last section, we prove the Proposition \ref{Main_Prop} by energy estimates. This closes our proof of Theorem \ref{Main_thm}.

{\bf Notations}  The set of all Schwartz functions is called a Schwartz space and is denoted $\mathcal{S}(\R^n)$. For any $x\in\R$, denote $\<x\>:=\sqrt{1+x^2}$. For any two numbers $A$ and $B$ and a absolute constant $C$, we denote
\begin{equation*}
	A\lesssim B\,,\ B\gtrsim A\,,\ \mathrm{if}\ A\leq CB\,.
\end{equation*}

For a vector $(v_1,v_2)\in \R^2$, we denote 
\begin{equation*}
	v^\bot=(-v_2,v_1)\,.
\end{equation*}
Here we always define $\om$ as the unit vector
\[  \om=\big(\frac{x_1}{r},\frac{x_2}{r}\big)=\frac{x}{r}\, .  \]

\bigskip
\section{Variational Derivation and reformulation}\label{sec-derivation}
In this section, we concern on the derivation and reformulation of the simplified Ericksen-Leslie's hyperbolic liquid crystal model \eqref{ori_sys}. First we will derive the 
system \eqref{ori_sys} with general target manifold for map $d$ by energetic variational approach. Then we will reformulate the
system in two different ways. One way is to use the angles between $d$ and axis or plane, which is convenient to study the wave map equations with sphere target; the other way is to use gauge theory, which is more suitable for general wave map equations.

\subsection{Derivation of Ericksen-Leslie's hyperbolic liquid crystal model}

Here we derive the Ericksen-Leslie's liquid crystal model by energy variational method. 

We consider the energy dissipation law:
\begin{align*}
	\frac{d}{dt}\int \frac{1}{2}\big(|u|^2-|D_t d|^2+|\nab d|^2\big)\  dx=-\int_\Om \nu |\nab u|^2\ dx\,,
\end{align*}
Then the action functional is given by
\begin{align*}
	\AA(x^\ep,d^\ep)=\frac{1}{2}\int \big( -|\d_t x^\ep|^2-|\d_t d^\ep|^2 +g^{\ep,ij}\d_i d^\ep \d_j d^\ep \big) \sqrt{\det g^\ep}\ dXdt \,,
\end{align*}
where the one-parameter families of such flow maps $x^\ep$ and $d^\ep$ satisfying
\begin{equation*}
    x^0(X,t)=x(X,t)\,, \quad d^0(X,t)=d(X,t)\,,\quad \frac{d}{d\ep}\big|_{\ep=0}x^\ep=\de x\,,\quad \frac{d}{d\ep}\big|_{\ep=0} d^\ep=\de d\in T_{}N.
\end{equation*}
Denote $F_\ep=\frac{\d x^\ep}{\d X}$, we have
\begin{align*}
	g^\ep_{ij}=\d_i x^\ep\cdot  \d_j x^\ep\,.
\end{align*}
For shorten of notation, we would use the notation $\de=\frac{d}{d\ep}\big|_{\ep=0}$ denoted as the variation/derivative.
Then 
\begin{align*}
	\de g_{ij} =\d_i \de x\cdot \d_j x+\d_i x\cdot \d_j \de x\,,
\end{align*}
which implies
\begin{align*}
	\de g^{ij} =-g^{ik}g^{jl}\de g_{kl}=-g^{ik}g^{jl}( \d_k \de x\cdot \d_l x+\d_k x\cdot \d_l \de x)\,.
\end{align*}

By imcompressibility we have
\begin{align*}
	\de \int -\frac{1}{2}|\d_t x^\ep|^2 \det F_\ep \ dXdt&= \int-\d_t x \d_t \de x \det F- \frac{1}{2}|\d_t x|^2 \tr (F^{-1}\nab \de x)\det F\ dX dt\\
	&=\int  \big(\d^2_t x +\d_t x \tr (F^{-1}\d_t F) \big) \cdot\de x\det F \ dX dt-\int \frac{1}{2} |u|^2 \nab\cdot \de x dxdt\\
	&=\int (D_t u  +\frac{1}{2}\nab\cdot |u|^2)\cdot \de x \ dx dt\,.
\end{align*}
and
\begin{align*}
	&\de\int \frac{1}{2}( -|\d_t d^\ep|^2 +g^{\ep,ij}\d_i d^\ep \d_j d^\ep ) \sqrt{\det g^\ep}\ dXdt\\
	&= \int   ( -\d_t d  \d_t \de d+\frac{1}{2}\de g^{ij} \d_i d \d_j d+g^{ij}\d_i d \d_j \de d )\sqrt{\det g}\\
	&\quad +\frac{1}{2} ( -|\d_t d|^2 +g^{ij}\d_i d \d_j d ) \det F \tr(F^{-1}\nab \de x)\ dXdt\\
	&=\int \Big( \d^2_t d  +\d_t d \tr (F^{-1}\d_t F) -\frac{1}{\sqrt{\det g}}\d_j (\sqrt{\det g} g^{ij}\d_i d)\Big) \de d \sqrt{\det g}\ dXdt\\
	&\quad +\int \frac{1}{2}\de g^{ij} \d_i d \d_j d \det F+\frac{1}{2} ( -|\d_t d|^2 +g^{ij}\d_i d \d_j d ) \det F \tr(F^{-1}\nab \de x)\ dXdt\\
	:&=I+II\,.
\end{align*}
The second integral $II$ can be written as
\begin{align*}
	&\int \frac{1}{2}\de g^{ij}\d_i d\d_j d \det F+\frac{1}{2} ( -|\d_t d|^2 +g^{ij}\d_i d \d_j d ) \det F \tr(F^{-1}\nab \de x)\ dXdt\\
	&=\int -g^{ik}g^{jl}( \d_k \de x\cdot \d_l x)\d_i d\d_j d \det F+\frac{1}{2} ( -|\d_t d|^2 +g^{ij}\d_i d \d_j d ) \det F \tr(F^{-1}\nab \de x)\ dXdt\\
	&=\int -\d_{X_\mu} d\nab_X d \d_{X_\mu}\de X   +\frac{1}{2} ( -|D_t d|^2 +|\nab d|^2 )  \nab_x \de x \  dx dt\\
	&=\int \d_{x_\mu}(\d_{x_\mu} d\nab_x d )\de x  -\frac{1}{2}\nab_x ( -|D_t d|^2 +|\nab d|^2 )  \de x \  dxdt\,.
\end{align*}

Here the relation in term of $d$ is obtained,
\begin{equation*}
	\big( \d^2_t d  +\d_t d \tr (F^{-1}\d_t F) -\frac{1}{\sqrt{\det g}}\d_j (\sqrt{\det g} g^{ij}\d_i d)\big) \de d=0\,.
\end{equation*}
Namely, 
\begin{align*}
	\d^2_t d+\d_t d \tr (F^{-1}\d_t F) -\De_g d=(-|\d_t d|^2+|\nab d|_g^2)d\,.
\end{align*}

As a consequence of the above computation, in the Eulerian coordinates, we obtain the equations of velocity $u$
\begin{align*}
	D_t u-\nu\De u+\nab p=-\d_j (\nab d\cdot \d_j d)\,.
\end{align*}
and the equation of orientation fields $d$ with $|d|=1$
\begin{align*}
	D_t^2 d-\De d = (-|D_t d|^2+|\nab d |^2)d\,.
\end{align*}
More generally, when the target manifold of map $d$ is $\mathcal N$, we have
\begin{align*}
	(D_t^2-\De) d^i = \Ga^i_{jk}(d)D^\al d^j D_\al d^k\,.
\end{align*}
where $\Ga^i_{jk}$ are the Christoffel symbols, and $D^t =-D_t$, $D_j=\d_{x_j}$ for $j=1,\cdots,n$.

\subsection{Reformulations of Ericksen-Leslie's liquid crystal model}
Here we will provide two formulations of model \eqref{ori_sys}, which are convenient to study this system.

First, we provide the standard formulation of \eqref{ori_sys} by gauge theory. We define the tangent vector-field $d^\bot$ in tangent bundle $T \S^1$ as 
\begin{align*}
	d^{\bot}=(-d_2,d_1)\,.
\end{align*}
We also define the main variables 
\begin{equation*}
	\psi_\al =\d_\al d\cdot d^\bot\,,\quad \al=0,\cdots ,n\,.
\end{equation*}
then the compatibility condition is imposed
\begin{equation*}
	\d_\al \psi_\be=\d_\be \psi_\al\,.
\end{equation*}
	
In view of the $d$-equation in \eqref{ori_sys}, we have
\begin{align*}
	D_t\big( (\psi_0+u\cdot\psi)d^\bot\big) -\d_j (\psi_j d^\bot)&=(-|\psi_0+u\cdot\psi|^2+|\psi|^2) d\,,
\end{align*}
which, due to the fact that $d\in \S^1$ and the expression of $d^\bot$, can be further written as
\begin{align*}
	D_t (\psi_0+u\cdot\psi) d^\bot-|\psi_0+u\cdot \psi|^2 d -\d_j \psi_j d^\bot +|\psi|^2 d &=(-|\psi_0+u\cdot\psi|^2+|\psi|^2) d\,.
\end{align*}
This yields
\begin{equation}     \label{wave}
	D_t (\psi_0+u\cdot\psi) -\d_j \psi_j=0\,.
\end{equation}

In order to derive a wave equation for variables $\psi_k$, we apply $\d_k$ to  \eqref{wave} and get
\begin{align*}
	0&=\d_k \big(D_t (\psi_0+u\cdot\psi) -\d_j \psi_j\big)\\
	&= D_t \d_k (\psi_0+u\cdot\psi)+\d_k u\cdot \nab (\psi_0+u\cdot \psi)-\De \psi_k\\
	&= D_t^2 \psi_k+D_t (\d_k u\cdot \psi)+\d_k u\cdot D_t \psi +\d_k u\cdot \nab u\cdot \psi-\De \psi_k\,.
\end{align*}
Thus we obtain the equation
\begin{equation*}
	D_t^2 \psi_k-\De \psi_k=-D_t (\d_k u\cdot \psi)-\d_k u\cdot D_t \psi -\d_k u\cdot \nab u\cdot \psi\,.
\end{equation*}

In conclusion, the system \eqref{ori_sys} is rewritten as 
\begin{equation}                   \label{Main_Sys-re2}
	\left\{\begin{aligned}
		&\partial_t u-\De u    =-u\cdot\nabla u-\nab p-\d^j(\psi \psi_j)\,, \\
		&\div\ u  =0\,,\\
		&D_t^2 \psi_k-\De \psi_k=-D_t (\d_k u\cdot \psi)-\d_k u\cdot D_t \psi -\d_k u\cdot \nab u\cdot \psi\,,
	\end{aligned}
	\right.
\end{equation}
with the constraints
\begin{equation*}
	\d_\al \psi_\be=\d_\be \psi_\al\,,\quad \al,\be=0,1,\cdots,n\,.
\end{equation*}
Note that the standard gauge formulation \eqref{Main_Sys-re2} is complicated, which may not be suitable for our problem.

\

Next, we state the second formulation in terms of angle $\phi$ due to the target manifold $\S^1$. Due to special target $\S^1$, the orientation field $d$ can be expressed by the angle $\phi$ between $d$ and $x$-axis:
\begin{align*}
	d(x,t)=(\cos \phi(x,t),\sin\phi(x,t))^T\,.
\end{align*}
Denote $d^\bot=(-\sin\phi,\cos\phi)$, then the wave map equation in \eqref{ori_sys} reads as 
\begin{align*}
	 d^\bot D^2_t\phi-|D_t \phi|^2 d-d^\bot \De \phi+|\nab \phi|^2 d=(-|D_t \phi|^2+|\nab\phi|^2) d\,,
\end{align*}
which implies 
\begin{align*}
	(D_t^2\phi-\De \phi) d^\bot=0\,.
\end{align*}
As a result, \begin{align*}
	D_t^2\phi-\De \phi=0\,.
\end{align*}
Meanwhile, the nonlinear term $-\d_j(\nab d \cdot\d_j d)$ can be written as 
\begin{align*}
	-\d_j (\nab d\cdot  \d_j d)=-\d_j (d^{\bot}\nab\phi\cdot d^{\bot}\d_j \phi)=-\d_j (\nab \phi\cdot  \d_j \phi)\,.
\end{align*}

Hence, the Ericksen-Leslie's hyperbolic liquid crystal model \eqref{ori_sys} in dimensions two reads 
\begin{equation}                   \label{Main_Sys-re}
	\left\{\begin{aligned}
		&\partial_t u-\De u    =-u\cdot\nabla u-\nab p-\d^j(\nab \phi\cdot \d_j \phi)\,, \\
		&\div\ u  =0\,,\\
		&(\d_t+u\cdot\nab)^2 \phi-\Delta\phi=0\,.
	\end{aligned}
	\right.
\end{equation}
We can find that the nonlinearity in $\phi$-equation is eliminated, which is very important in understand the system. The above formulation \eqref{Main_Sys-re} is really effective when the target is a circle $\S^1$.

\medskip

Compared the above two formulations \eqref{Main_Sys-re2} and \eqref{Main_Sys-re}, when the target manifold is $\S^1$, we would use the second formulation \eqref{Main_Sys-re} to study the long-time behavior of hyperbolic liquid crystal, which is a quite perfect formula. However, the acoustical metric is still present, where the investigation of the structure is the main part of the article.

\bigskip 
\section{Preliminaries and the main proposition}  \label{sec-pre}
In this section, we introduce the vector fields and define the function spaces. Then we state our main bootstap proposition, which  will be used to prove Theorem \ref{Main_thm}.

\subsection{Vector fields and function spaces}
Here we start with the vector fields, which are used to define the function spaces. From the rotation invariance, we define the perturbed angular momentum operators by
\begin{equation*}
	\tilde{\Om} u=\Om u+A u\,,\quad  \tilde{\Om}_i d=\Om_i d\,,
\end{equation*}
where $\Om=x_1\d_2-x_2\d_1=\d_\th$ is the rotation vector-field and the matrix $A$ is defined by
\begin{equation*}
	A=\left(\begin{array}{ccc}
		0&1\\
		-1&0
	\end{array}\right)\,.
\end{equation*}

Next, we introduce the vector fields corresponding to the scaling invariance. Denote the standard scaling vector field $S$ as 
\begin{equation*}
	S=t\d_t+x_i\d_{x_i}\,.
\end{equation*}
We note that the system \eqref{Main_Sys} does not have any scaling invariance. However, if the viscosity disappear, we will have the scaling transform: for any $\la>0$,
\begin{align*}
	u_{\la}(t,x)=u(\la t,\la x)\,,\qquad p_\la(t,x)=p(\la t,\la x)\,,\qquad \phi_{\la}(t,x)=\la^{-1}\phi(\la t,\la x)\,.
\end{align*}
Inspired by these, we denote the perturbed scaling operator as 
\begin{equation*}
	\tilde{S}u=Su\,,\quad \tilde{S}p=Sp,\quad \tilde{S}\phi=(S-1)\phi\,.
\end{equation*}
Although applying $S$ to the equation of $u$ would produce lower-order terms, they still can be dealt with in the decay estimates and energy estimates. Hence  the scaling operators play key roles under this circumstance.

For convenience, we denote
\begin{equation*}
	Z\in \{\d_t,\d_1,\d_2,\tilde{\Om}\}\,,\qquad \Ga^\al=\tilde{S}^{\al_1}Z^{\al'}\,,
\end{equation*}
where $\al=(\al_1,\al'):=(\al_1,\al_2,\cdots,\al_5)\in\Z_+^5$ are multi-indices and $Z^{a'}=Z^{\al_2}Z^{\al_3}\cdots Z^{\al_5}$.

Applying vector fields $Z$ to the system (\ref{Main_Sys}) we can derive yields
\begin{equation}         \label{Main_Sys_VecFie}
	\left\{\begin{aligned}
		&\d_t \Ga^\al u-\De(S-1)^{\al_1} Z^{\al'}u+\nab \Ga^\al p=f_\al\,,\\
		&\div\ \Ga^\al u=0\,,\\
		&\d_t^2\Ga^\al \phi-\Delta\Ga^\al\phi=g_\al\,,
	\end{aligned}
	\right.
\end{equation}
where nonlinearities $f_\al$ and $g_\al$ are
\begin{align}  \label{fal}
	f_\al&=-\sum_{\be+\ga=\al}C_\al^\be\big( \Ga^\be u\cdot\nab \Ga^\ga u+ \d_j(\Gamma^{\be}\phi\cdot\d_j\Gamma^{\ga}\phi)\big)\,,\\  \label{gal}
	g_\al&=-\sum_{\be+\ga=\al}C_\al^\be(\d_t \Ga^\be u\cdot\nab\Ga^\ga\phi+2\Ga^\be u\cdot\nab\d_t\Ga^\ga\phi)
	-\sum_{\be+\ga+\si=a}C_\al^{\be,\si}\Ga^\be u\cdot\nab(\Ga^\ga u\cdot\nab\Ga^\si \phi)\,,
\end{align}
and the constants are
\begin{equation*}
C_a^b:=\frac{a!}{b!(a-b)!}\,,\qquad  
C_a^{b,c}:=\frac{a!}{b!c!(a-b-c)!}\,.
\end{equation*}

We define the Klainerman’s generalized energy by
\begin{align*}
E_{N}(t)=\sum_{|\al|\leq N} \Big(\|\Ga^\al u\|_{L^2}^2+\int_0^t  \|\nab\Ga^\al u\|_{L^2}^2\ ds+\|\nab_{t,x}\Ga^\al\phi\|_{L^2}^2\Big)\,,
\end{align*}
and define the weighted $L^2$ generalized energy by
\begin{align}         \label{XX}
\XX_{N}(t)=\sum_{|\al|\leq N-1}\Big( \int_{r\leq 2\<t\>} \<t-r\>^2 |D^2\Ga^\al\phi|^2\ dx+ \int_{r> 2\<t\>} \<t\>^2 |D^2\Ga^\al\phi|^2\ dx \Big)\,.
\end{align}
Here we also need the dissipative energy and ghost weight energy given by
\begin{equation*}
\mathcal D_{N}(t)=\sum_{|\al|\leq N}\Big(\|\nab\Ga^\al u\|_{L^2}+\Big\|\frac{\Ga^\al u}{\<t-r\>}\Big\|_{L^2}+ \Big\|\frac{\om \d_t \Ga^\al \phi+\nab\Ga^\al \phi}{\<t-r\>}\Big\|_{L^2}\Big)\,,
\end{equation*}
where the second and third terms are called \emph{ghost weight energy} that comes from the modified energy estimates.

In order to characterize the initial data, we introduce the time independent
analogue of vector field $Z$. The only difference will be in the scaling operator. Set
\begin{equation*}
	\Lambda=(\Lambda_1,\cdots,\Lambda_4)=(\d_1,\d_2,\tilde \Om,\tilde S_0)\,,\quad \tilde S_0=\tilde S-t\d_t\,.
\end{equation*}
Then the commutator of any two $\Lambda$'s is again a $\Lambda$. We define 
\begin{equation*}
	H^N_\Lambda=\Big\{ (u,\phi_0,\phi_1):\sum_{|\al|\leq N}\big(\|\Lambda^\al u\|_{L^2}+\|\nab \Lambda^\al \phi_0\|_{L^2} +\| \Lambda^\al \phi_1\|_{L^2}\big) <\infty \Big\}\,.
\end{equation*}
We study the liquid crystal model in the space
\begin{align*}
	H^N_\Ga(T)=\Big\{ (u,\phi): \Ga^\al u, \d_t \Ga^\al\phi,\nab \Ga^\al\phi\in L^\infty([0,T]; L^2), \nab \Ga^\al u\in L^2([0,T];L^2) ,\quad \forall |\al|\leq N     \Big\}\,.
\end{align*}

\subsection{The bootstrap proposition and the proof of Theorem \ref{Main_thm}}
Our main result is the following proposition:
\begin{proposition}[Bootstrap proposition]   \label{Main_Prop}
	Assume that $(u,\phi)$ is a solution to \eqref{Main_Sys} on some time interval $[0,T]$ for any $T\leq e^{C_T/\ep}$ with initial data satisfying the assumptions \eqref{MainAss_dini}. Assume also that the solution satisfies the bootstrap hypothesis
	\begin{gather}           \label{Main_Prop_Ass1}
		E^{1/2}_{N}(t)\leq C_0\ep\,,\qquad \XX^{1/2}_{N-2}\leq C_0^2\ep\,.
	\end{gather}
	Then the following improved bounds hold
	\begin{gather}\label{Main_Prop_result1}
		E^{1/2}_{N}(t)\leq \frac{C_0}{2}\ep\,,\qquad \XX^{1/2}_{N-2}\leq \frac{C_0^2}{2}\ep\,.
	\end{gather}
\end{proposition}

With the bootstrap proposition at hand, we can prove the main Theorem \ref{Main_thm} as below. The next sections will focus on the proof of this proposition.

\begin{proof}[Proof of Theorem \ref{Main_thm}]	By local well-posedness of \eqref{Main_Sys} in \cite{Jiang-Luo-2018} we make a bootstap assumption where we assume the same bounds but with a worse constant, as follows:
	\begin{align*}
		E^{1/2}_N(t)\leq C_0\ep\,,\quad \XX^{1/2}_{N-2}(t)\leq C_0^2 \ep\,,
	\end{align*}
	on the time interval $[0,T]$ for $T\leq e^{C_T/\ep}$. Here $C_0\geq 1$ and $C_T$ are determined in Proposition \ref{Main_Prop}.
	
	Choosing $\ep$ sufficiently small, by Proposition \ref{Main_Prop}, we get the improved bounds
	\begin{align*}
		E^{1/2}_N(t)\leq \frac{C_0}{2}\ep\,,\quad \XX^{1/2}_{N-2}(t)\leq \frac{C_0^2}{2} \ep\,.
	\end{align*}
	Then the solution can be extended forward until $t\sim e^{C_T/\ep}$.
	Hence, the system \eqref{Main_Sys} with small initial data \eqref{MainAss_dini} admits a unique solution on a time interval of order $T\sim e^{C_T/\ep}$, whose energy $E_N^{1/2}$ is bounded by $C_0\ep$. The decay estimates \eqref{decay-thm} are obtained from \eqref{Decay-u}, \eqref{AwayCone0} and \eqref{NearCone0}, i.e. 
	\begin{align*}
		\|u\|_{\Lf}\les \<t\>^{-1/2}(E^{1/2}_2+\XX^{1/2}_2)\les \ep\<t\>^{-1/2}\,,
	\end{align*}
	and 
	\begin{align*}
		\|\nab_{t,x}\phi\|_{\Lf}\les \<t\>^{-1/2}(E^{1/2}_1+\XX^{1/2}_2)+\<t\>^{-1+\de}(E^{1/2}_0+\XX^{1/2}_2)\les \ep\<t\>^{-1/2}\,.
	\end{align*}
	Thus the proof of Theorem \ref{Main_thm} completes.
\end{proof}

\bigskip

\section{Decays Estimates}  \label{sec-decay}
In this section, we give various decay estimates of $u$ and $\phi$, which are bounded by  the energy $E_k$ and the weighted $L^2$ generalized energy $\XX_k$ for some $k\leq N$.  First, we prove the decay estimates for $D\phi$ and $D^2\phi$ in different regions. Then we prove the decay estimates of velocity $u$ by Duhamel's formulas and the decays of $\phi$. These estimates will be needed in the next sections in order for the control of $\XX_k$ and the energy estimates. 

\subsection{Decay estimates of the angle \texorpdfstring{$D\phi$}{}.} In order for the decay of $D\phi$, we start with the following weighted $L^{\infty}-L^2$ estimate. 

\begin{lemma}Let $f\in H^2(\R^2)$, $r=|x|$ and $\la=\frac{1}{2}\ or\ 1$, then there holds
	\begin{align}   \label{rv}
		&\lV  r^\la f\rV_{\Lf}\lesssim \sum_{|\alpha|\leq 1}(\lV r^{2\la-1}\d_r\Om^{\alpha}f\rV_{L^2}+\lV\Om^{\alpha}f\rV_{L^2})\,,\\    \label{rsi-f}
		&\| r^{1/2}\<t-r\>^\la f\|_{\Lf}\les \sum_{|\al|\leq 1 }\|\<t-r\>\d_r \Om^\al f\|_{L^2}+\|\<t-r\>^{2\la-1}\Om^\al f\|_{L^2}\,.
	\end{align}
	In particular, if $\div f=0$ for a vector $f=(f_1,f_2)$, then 
	\begin{equation}  \label{omf}
		\|r(\om\cdot f)\|_{L^\infty}\lesssim \sum_{|\alpha|\leq 2}\lV \Om^{\alpha}f\rV_{L^2}\,.
	\end{equation}
\end{lemma}
\begin{proof}
	The first two estimates  are standard. Readers can refer to \cite[Lemma 1]{Si97} and \cite[Lemma 3.3]{LSZ} for the proof. 
	
We prove the bound \eqref{omf}. By $\div f=0$ and the decomposition
	\begin{equation*}
		\nab=\frac{x}{r}\d_r+\frac{x^{\bot}}{r^2}\Om=\om\d_r+\frac{\om^\bot}{r}\d_\th\,,
	\end{equation*}
	we have
	\begin{equation*}
		0=\div\ f=\om\cdot \d_r f+\frac{x^{\bot}}{r^2}\Om \cdot f=\d_r(\om \cdot f)+\frac{x^{\bot}}{r^2}\Om \cdot f\,.
	\end{equation*}
	Combinning with \eqref{rv}, we obtain
	\begin{align*}
		\|r(\om\cdot f)\|_{L^\infty}&\lesssim \sum_{|\alpha|\leq 1}\big(\lV r\d_r\Om^{\alpha}(\om\cdot f)\rV_{L^2}+\lV\Om^{\alpha}(\om\cdot f)\rV_{L^2}\big)\\
		&=\sum_{|\alpha|\leq 1}\big(\lV \Om^{\alpha}r\d_r(\om\cdot f)\rV_{L^2}+\lV\Om^{\alpha}(\om\cdot f)\rV_{L^2}\big)\\
		&=\sum_{|\alpha|\leq 1}\big(\lV \Om^{\alpha}(\frac{x^{\bot}}{r}\Om\cdot f)\rV_{L^2}+\lV\Om^{\alpha}(\om\cdot f)\rV_{L^2}\big)\\
		&\lesssim \sum_{|\alpha|\leq 2}\lV \Om^{\alpha}f\rV_{L^2}
	\end{align*}
and \eqref{omf} follows.
\end{proof}

Next we derive the basic decay estimates for $D\phi$ in different regions.
\begin{lemma} \label{Decay-phi0}
	Let $t>4$. Then there holds
	\begin{align}     \label{r1/2phi}
		&\<r\>^{1/2}|D\Ga^\al\phi|\lesssim \|D\Ga^{\leq 2}\Ga^\al\phi\|_{L^2}\lesssim E_{|\al|+2}^{1/2}\,,\\\label{phi-lf}
		&\|\Om \Ga^\al\phi\|_{L^\infty}\les E_{|\al|+2}^{1/2}\,.
	\end{align}
    Moreover, when $r$ stays away from light cone: $r\leq \frac{2}{3}\<t\>$ or $r\geq \frac{5}{4}\<t\>$, there holds
	\begin{equation}  \label{AwayCone0}
		t|D\Ga^\al \phi|\lesssim \mathcal (E^{1/2}_{|\al|}+\XX^{1/2}_{|\al|+2})\ln^{1/2}(e+t) \,.
	\end{equation}
	When $r$ stays near the light cone: $\frac{\<t\>}{3}\leq r\leq \frac{5}{2}\<t\>$, there holds 
	\begin{equation}   \label{NearCone0}
		\<r\>^{1/2}\<t-r\>^{1/2}|D\Ga^\al\phi|\lesssim  \mathcal E^{1/2}_{|\al|+1}+\mathcal X^{1/2}_{|\al|+2}\,.
	\end{equation}
\end{lemma}
\begin{proof}
	The first bound \eqref{r1/2phi} follows from \eqref{rv}.
	For \eqref{NearCone0}, readers can refer to \cite[Lemma 1]{Si97}. It remains to prove the bound \eqref{phi-lf} and \eqref{AwayCone0}. 
	
	Due to the expression of $\Om=\d_\theta$ and \eqref{rv}, we have
	\begin{align*}
		\|\Om \Ga^\al\phi\|_{L^\infty}&=\big\|r\big(\frac{\d_\th}{r} \Ga^\al\phi\big)\big\|_{L^\infty}\\
		&\les \sum_{|\de|\leq 1}\big(\big\|r\d_r\Om^{\de}\big(\frac{\d_\th}{r}\Ga^\al\phi\big)\big\|_{L^2}+\big\|\Om^\de \big(\frac{\d_\th}{r}\Ga^\al\phi \big)\big\|_{L^2}\big)\\
		&\les \sum_{|\de|\leq 1}\big(\big\|\d_r \big(r\frac{\d_\th}{r}\Om^{\de}\Ga^\al\phi \big)\big\|_{L^2}+\big\|\frac{\d_\th}{r}\Om^\de\Ga^\al\phi \big\|_{L^2}\big)\\
		&\les \sum_{|\de|\leq 1}\big(\|\d_r(\Om^{\de+1}\Ga^\al\phi)\|_{L^2}+\|\nab\Om^\de\Ga^\al\phi\|_{L^2}\big)\\
		&\les E_{|\al|+2}^{1/2}\,.
	\end{align*}
Then the bound \eqref{phi-lf} is obtained.

Finally, multiplying $D\Ga^\al\phi$ by a smooth radial cut-off function supported in $r<3\<t\>/4,\ r>6\<t\>/5$, then the bound \eqref{AwayCone0} follows from \eqref{XX} and the estimate
\begin{align*}
	\|f\|_{\Lf}\les \ln^{1/2}(e+t)\|\nab f\|_{L^2}+\<t\>^{-1}(\|f\|_{L^2}+\|\nab^2 f\|_{L^2})\,,
\end{align*}
which is easily obtained by Littlewood-Paley decomposition and Bernstein's inequality. For the detail, we can refer to \cite[Lemma 3.1]{Lei2016}.
\end{proof}

\medskip 

We also need the decay estimates of the second-order derivatives of angle in the $L^\infty$ norm. 

\begin{lemma} \label{Decay-phi}
	Let $t>4$. Then for $r$ away from light cone: $r\leq \frac{2}{3}\<t\>$ or $r\geq \frac{5}{4}\<t\>$, there hold
	\begin{equation}  \label{AwayCone}
		t|D^2\Ga^\al \phi|\lesssim \mathcal X^{1/2}_{|\al|+3}\,.
	\end{equation}
    For $r$ near the light cone: $r\leq \frac{5}{2}\<t\>$, there holds 
    \begin{equation}   \label{NearCone}
    	\<r\>^{1/2}\<t-r\>|D^2\Ga^\al\phi|\lesssim  E^{1/2}_{|\al|+2}+ \mathcal X^{1/2}_{|\al|+3}\,.
    \end{equation}
\end{lemma}
\begin{proof}
The first bound \eqref{AwayCone} is obtained by Sobolev embeddings and \eqref{XX}. The second bound \eqref{NearCone} is a consequence of \eqref{rsi-f} with $\la=1$. We can refer to \cite[Lemma 3.2]{Lei2016} for the detail.
\end{proof}

\medskip
\subsection{Decay estimates of velocity \texorpdfstring{$u$}{}}
Here we prove some decay estimates of velocity $u$. First, we prove the basic decay estimates of $\Ga^\al u$ for any $|\al|\leq N-2$ by using Duhamel's formula and the decays of $\nab\Ga^\al\phi$. They play an important role in the proof of other decay estimates. Second, the decay estimates of $\d_t \Ga^\al u$ and $\nab\Ga^\al u$ in $L^2$ are provided. Finally, we prove the decay of $\nab\Ga^\al u$ in $\Lf$ which, combined with the basic decay estimates, yields the decay of $\nab^2 \Ga^\al u$. And hence we also obtain the decay of $\d_t \Ga^\al u$ in $\Lf$.  

The decay estimates are as follows:

\begin{proposition}[Decay estimates of $u$]\label{Dec-lem}
Let $N\geq 10$ and $\de=1/100$. Assume that $(u,\phi)$ is the solution of \eqref{Main_Sys} satisfying \eqref{MainAss_dini} and \eqref{Main_Prop_Ass1}. Then the following decay estimates hold:

i) (Basic decay estimates) For any $|\al|\leq N-2$, we have
\begin{align}   \label{Decay-u}
\|\<\nab\>^{(N-|\al|)\de}\Ga^{\al} u\|_{L^\infty}&\lesssim \<t\>^{-\frac{1}{2}}\ln^{|\al|}(e+t)(E^{1/2}_{|\al|+2}+\XX^{1/2}_{|\al|+2})\,.
\end{align}

ii) (Decays in $L^2$) For any $|\al|\leq N-1$, we have
\begin{align}\label{dtu-L2al}
	\|\d_t\Ga^\al u\|_{L^2}&\lesssim \|\nab\Ga^{ |\al|+1}u\|_{L^2}(1+E^{1/2}_{|\al|+1})+ \<t\>^{-\frac{1}{2}}(E^{1/2}_{[|\al|/2]+2}+\XX^{1/2}_{[|\al|/2]+3})\,,
\end{align}
and for any $|\al|\leq N-2$, we have
\begin{align}\label{du-L2}
	\|\<\nab\>^{(N-|\al|)\de}\nab\Ga^\al u\|_{L^2}\lesssim \<t\>^{-\frac{1}{2}}\ln^{|\al|+1}(e+t)(E^{1/2}_{|\al|+2}+\XX^{1/2}_{|\al|+2})\,.
\end{align}
As a consequence of these estimates, for any $|\al|\leq N-3$, we have
\begin{equation}	\label{dtu-L2}
	\|\d_t\Ga^\al u\|_{L^2}\lesssim \<t\>^{-\frac{1}{2}}\ln^{|\al|+2}(e+t)(E^{1/2}_{|\al|+3}+\XX^{1/2}_{|\al|+3})\,.
\end{equation}

iii) (Decays in $\Lf$) For any $|\al|\leq N-3$, we have
\begin{align}\label{decay-du0}
	&\|\<\nab\>^{(N-|\al|)\de}\nab \Ga^\al u\|_{L^\infty}\lesssim \<t\>^{-1}\ln^{|\al|+1}(e+t) (E^{1/2}_{|\al|+3}+\XX^{1/2}_{|\al|+3}) \,.
\end{align}
For any $|\al|\leq N-4$, we have the improved estimates
\begin{align}\label{decay-d2Gau}
	\|\nab^{2} \Ga^\al u\|_{L^\infty}\lesssim \<t\>^{-1} (E^{1/2}_{|\al|+4}+\XX_{|\al|+4}^{1/2})\,.
\end{align}
From this estimate, there holds for any $|\al|\leq N-4$ that
\begin{align}\label{decay-dtu}
	\|\d_t \Ga^\al u\|_{L^\infty}\lesssim \<t\>^{-1} (E^{1/2}_{|\al|+4}+\XX_{|\al|+4}^{1/2})\,.
\end{align}
\end{proposition}

\bigskip 
We will prove these estimates step by step. In the proofs,  the decay estimates of heat flow will be needed, which is standard,
	\begin{gather}        \label{Dec_heat}
		\lV e^{t\De}|\nab|^k f\rV_{L^q}\lesssim t^{-\frac{n}{2}(\frac{1}{p}-\frac{1}{q})-\frac{k}{2}} \lV f\rV_{L^p}\,,\qquad  1\leq p\leq q\leq \infty\,.
	\end{gather}
We also need the following Duhamel's formula
\begin{equation}     \label{Duhamel}
	\begin{aligned}
		\Ga^{\al}u(t)
		&=e^{t\De}\Ga^\al u(0)+\int_0^t e^{(t-s)\De}\Big[\De\sum_{\be_1=0}^{\al_1-1}C_{\al_1}^{\be_1}(-1)^{\al_1-\be_1}S^{\be_1}\Ga^{\al'}u(s)\\
		&\quad -\P\sum_{\be+\ga=\al}C_\al^\be[ \Ga^\be u\cdot\nab \Ga^\ga u+ \d_j(\nab\Gamma^{\be}\phi\cdot\d_j\Gamma^{\ga}\phi)]\Big]\ ds\,,
	\end{aligned}
\end{equation}
which is derived from the equations of $\Ga^\al u$ in \eqref{Main_Sys_VecFie}.
First we prove the basic decay estimates \eqref{Decay-u}.

\bigskip 
\begin{proof}[Proof of the estimate \eqref{Decay-u}]
	\ 
	
	Due to the presence of lower order terms in $\Ga^\al u$-equations \eqref{Main_Sys_VecFie}, we prove \eqref{Decay-u} by induction. Assume that for any $|\be|<|\al|$,
	\begin{align}\label{u-ass1}
		\<t\>^{1/2}\ln^{-|\be|}(e+t)\|\<\nab\>^{(N-|\be|)\de}\Ga^\be u\|_{L^\infty}\lesssim  E^{1/2}_{|\be|+2}+\XX^{1/2}_{|\be|+2}\,.
	\end{align}
	
	Applying $\<\nab\>^{(N-|\al|)\de}$ to the Duhamel's formula \eqref{Duhamel}, we have
	\begin{align*}
		\|t^{1/2}\<\nab\>^{(N-|\al|)\de} \Ga^\al u\|_{\Lf}\les  \KK_0+\KK_1+\KK_2\,,
	\end{align*}
	where 
	\begin{align*}
		\KK_0=&\|t^{1/2}e^{t\De}\<\nab\>^{(N-|\al|)\de}\Ga^\al u(0)\|_{\Lf}\,,\\
		\KK_1=&\sum_{|\be|<|\al|}\|t^{1/2}\int_0^t e^{(t-s)\De}\<\nab\>^{(N-|\al|)\de}\De\sum_{\be_1=0}^{\al_1-1}\Ga^{\be}u(s)\ ds\|_{\Lf}\,,\\
		\KK_2=&\sum_{\be+\ga=\al}\|t^{1/2}\int_0^{t}e^{(t-s)\De}   \<\nab\>^{(N-|\al|)\de}\P\big(\Ga^\be u \cdot\nab \Ga^\ga u+\nab(\nab\Ga^\be\phi \nab\Ga^\ga\phi)\big)(s)\ ds\|_{\Lf}\,.
	\end{align*}
	By \eqref{Dec_heat} the first term $\KK_0$ is bounded by
	\begin{equation*}
		\KK_0\lesssim (1+\<t\>^{-(N-|\al|)\de/2})\|\Ga^\al u(0)\|_{H^2}\lesssim E_{|\al|+2}^{1/2}(0)\,.
	\end{equation*}
	For the term $\KK_1$, by \eqref{Dec_heat} and \eqref{u-ass1} we have
	\begin{equation}   \label{KK1}
		\begin{aligned}
			\KK_1\lesssim& \int_0^{t/2} (t-s)^{-1}(1+\<t-s\>^{-(N-|\al|)\de/2})\|\Ga^\be u\|_{L^2} ds\\
			&+\int_{t/2}^{t-1} (t-s)^{-1}(1+\<t-s\>^{-(N-|\al|)\de/2})\|s^{1/2}\nab\Ga^\be u\|_{L^2} ds\\
			&+\int_{t-1}^t (t-s)^{-1+\de/2}\|s^{1/2}\<\nab\>^{(N-|\al|+1)\de}\nab\Ga^\be u\|_{L^2} ds\\
			\lesssim&\  E_{|\be|}^{1/2}+\ln^{|\be|+1}(e+t)(E^{1/2}_{|\be|+2}+\XX^{1/2}_{|\be|+2})\,.
		\end{aligned}
	\end{equation}
	
	Next, the term $\KK_2$ is divided to three parts
	\begin{align*}
		\KK_2\les& \sum_{\be+\ga=\al}\|t^{1/2}\int_0^{t/2}e^{(t-s)\De}   \<\nab\>^{(N-|\al|)\de}\P\big(\Ga^\be u \cdot\nab \Ga^\ga u+\nab(\nab\Ga^\be\phi \nab\Ga^\ga\phi)\big)(s)\ ds\|_{\Lf}\\
		&+\sum_{\be+\ga=\al}\|t^{1/2}\int_{t/2}^t e^{(t-s)\De}   \<\nab\>^{(N-|\al|)\de}\P(\Ga^\be u \cdot\nab \Ga^\ga u)(s)\ ds\|_{\Lf}\\
		&+\sum_{\be+\ga=\al}\|t^{1/2}\int_{t/2}^t e^{(t-s)\De}  \<\nab\>^{(N-|\al|)\de}\P\nab(\nab\Ga^\be \phi \nab\Ga^\ga\phi)(s)\ ds\|_{\Lf}\\
		:=&\ \KK_{21}+\KK_{22}+\KK_{23}\,,
	\end{align*}

	We use \eqref{Dec_heat} to bound $\KK_{21}$ by
    \begin{equation}    \label{K21}
    	\begin{aligned}
    		\KK_{21} 
    		\lesssim& \  \sum_{|\be+\ga|\leq |\al|}\int_0^{t/2}(t-s)^{-1}(1+(t-s)^{-(N-|\al|)\de/2}) (\|\Ga^\be u\cdot \Ga^\ga u\|_{L^1} +\|\nab\Ga^\be\phi \cdot\nab\Ga^\ga\phi\|_{L^1})\ ds\\
    		\lesssim&\  \sum_{|\be+\ga|\leq |\al|} (\|\Ga^\be u\cdot\Ga^\ga u\|_{L^\infty L^1} +\|\nab\Ga^\be\phi \cdot\nab\Ga^\ga\phi\|_{L^\infty L^1})\\
    		\lesssim&\  E^{1/2}_{|\be|}E^{1/2}_{|\ga|}\,.
    	\end{aligned}
    \end{equation}
	For the term $\KK_{22}$, by H\"older's inequality and \eqref{u-ass1} we have
	\begin{align*}
		\KK_{22} &\lesssim \sum_{|\be+\ga|\leq |\al|}\int_{t/2}^{t} (t-s)^{-1/2}(1+(t-s)^{-(N-|\al|)\de/2})  \|s^{1/2}\Ga^\be u \  \Ga^\ga u\|_{\Lf}\ ds\\
		&\les \|s^{1/2}\Ga^\be u\|_{\Lf(t/2,t;\Lf)}\|s^{1/2}\Ga^\ga u\|_{\Lf(t/2,t;\Lf)}\,.
	\end{align*}
	Finally, by \eqref{Dec_heat}, \eqref{AwayCone0} and \eqref{NearCone0} the term $\KK_{23}$ is treated as
	\begin{align*}
		\KK_{23}&\les \int_{t/2}^{t} (t-s)^{-1/2}\big(1+(t-s)^{-(N-|\al|)\de/2}\big)s^{1/2}\|\nab\Ga^\be\phi\nab\Ga^\ga \phi\|_{\Lf} ds\\
		&\les \|s^{1/2}\nab\Ga^\be\phi\|_{\Lf(t/2,t;\Lf)}\|s^{1/2}\nab\Ga^\ga \phi\|_{\Lf(t/2,t;\Lf)}\\
		&\les (E_{|\al|+1}+\XX_{|\al|+2})\,.
	\end{align*}

	Collecting the above estimates from $\KK_0$ to $\KK_2$ yields
	\begin{equation}   \label{u-bound}
		\begin{aligned}
			&\|t^{1/2}\ln^{-|\al|}(e+t)\<\nab\>^{(N-|\al|)\de}\Ga^\al u\|_{\Lf\Lf}\\
			&\les E^{1/2}_{|\al|+2}+\XX^{1/2}_{|\al|+2}
			+ \sum_{\be+\ga=\al}\|t^{1/2}\ln^{-|\be|}(e+t)\Ga^\be u\|_{\Lf\Lf}\|t^{1/2}\ln^{-|\ga|}(e+t)\Ga^\ga u\|_{\Lf \Lf}\,.
		\end{aligned}
	\end{equation}
    We note that, at initial time $t=0$, the bound for $\|\<\nab\>^{(N-|\al|)\de}\Ga^\al u(0)\|_{\Lf}\les E^{1/2}_{|\al|+2}$ is sufficiently small. Then for any $j\leq N-4$, we sum over $|\al|\leq k$ to yield
    \begin{align*}
    	&\sum_{|\al|\leq j}\|t^{1/2}\ln^{-|\al|}(e+t)\<\nab\>^{(N-|\al|)\de}\Ga^\al u\|_{\Lf\Lf}
    	\les E^{1/2}_{j+2}+\XX^{1/2}_{j+2}
    	+(\sum_{|\be|\leq j} \|t^{1/2}\ln^{-|\be|}(e+t)\Ga^\be u\|_{\Lf\Lf})^2\,.
    \end{align*}
By continuity method and the smallness of $E^{1/2}_N$ and $\XX^{1/2}_{N-2}$, we obtain
\begin{align*}
	\sum_{|\al|\leq j}\|t^{1/2}\ln^{-|\al|}(e+t)\<\nab\>^{(N-|\al|)\de}\Ga^\al u\|_{\Lf\Lf}\les E^{1/2}_{j+2}+\XX^{1/2}_{j+2}\,.
\end{align*}   
For any $N-3\leq j\leq N-2$, by \eqref{u-bound} and \eqref{u-ass1} we have
\begin{align*}
	&\sum_{|\al|= j}\|t^{1/2}\ln^{-|\al|}(e+t)\<\nab\>^{(N-|\al|)\de}\Ga^\al u\|_{\Lf\Lf}\\
	&\les (E^{1/2}_{j+2}+\XX^{1/2}_{j+2})+\sum_{|\be+\ga|=|\al|;|\be|,|\ga|<|\al|}(E^{1/2}_{|\be|+2}+\XX^{1/2}_{|\be|+2})(E^{1/2}_{|\ga|+2}+\XX^{1/2}_{|\ga|+2})\\
	&\quad+(E^{1/2}_2+\XX^{1/2}_2)\|t^{1/2}\ln^{-|\al|}(e+t)\Ga^\al u\|_{\Lf\Lf}\,.
\end{align*} 
Since $N\geq 10$ and \eqref{Main_Prop_Ass1}, the second term can be bounded by 
\begin{align*} 
	 \sum_{|\be+\ga|=|\al|;|\be|,|\ga|<|\al|}(E^{1/2}_{|\be|+2}+\XX^{1/2}_{|\be|+2})(E^{1/2}_{|\ga|+2}+\XX^{1/2}_{|\ga|+2})\les  (E^{1/2}_{[j/2]+2}+\XX^{1/2}_{[j/2]+2})(E^{1/2}_{j+1}+\XX^{1/2}_{j+1}) \les \ep (E^{1/2}_{j+1}+\XX^{1/2}_{j+1}) \,,  
\end{align*}
and the last term can be absorbed by the left hand side. Then we obtain 
\begin{align*}
	&\sum_{|\al|= j}\|t^{1/2}\ln^{-|\al|}(e+t)\<\nab\>^{(N-|\al|)\de}\Ga^\al u\|_{\Lf\Lf}
	\les (E^{1/2}_{j+2}+\XX^{1/2}_{j+2})\,.
\end{align*} 
Thus the desired bound \eqref{Decay-u} follows.
\end{proof}

\medskip 
Next, we prove the decay estimates of $\nab\Ga^\al u$ and $\d_t \Ga^\al u$ in $L^2$. 

\begin{proof}[Proof of the estimate \eqref{dtu-L2al}]
	
\ 

By the $\Ga^\al u$-equation in \eqref{Main_Sys_VecFie}, we have
\begin{align*}
	\|\d_t\Ga^\al u\|_{L^2}\lesssim \sum_{|\be|\leq |\al|}\|\De\Ga^\be u\|_{L^2}+\sum_{|\be+\ga|\leq |\al|}\|\P(\Ga^\be u\cdot\nab\Ga^\ga u)\|_{L^2}+\sum_{|\be+\ga|\leq |\al|}\|\P\nab(\nab\Ga^\be \phi\cdot\nab\Ga^\ga \phi)\|_{L^2}\,.
\end{align*}
For the second term, by Sobolev embedding we have
\begin{align*}
	\sum_{|\be+\ga|\leq |\al|}\|\P(\Ga^\be u\cdot\nab\Ga^\ga u)\|_{L^2}&\lesssim \|\Ga^{|\al|} u\|_{L^4}\|\nab\Ga^{|\al|} u\|_{L^4}
	\lesssim E^{1/2}_{|\al|+1}\|\nab\Ga^{ |\al|+1} u\|_{L^2}\,.
\end{align*}
By \eqref{AwayCone0} and \eqref{NearCone0} we bound the last term by
\begin{align*}
	\sum_{|\be+\ga|\leq |\al|}\|\P\nab(\nab\Ga^\be \phi\cdot\nab\Ga^\ga \phi)\|_{L^2}
	&\les E^{1/2}_{|\al|+1}\|\nab\Ga^{|\al|}\phi\|_{\Lf}
	\lesssim  \<t\>^{-1/2}(E^{1/2}_{[|\al|/2]+2}+\XX^{1/2}_{[|\al|/2]+3})E^{1/2}_{|\al|+1}\,.
\end{align*}
The above bounds give the estimate \eqref{dtu-L2al}. This completes the proof.
\end{proof}

\medskip 
\begin{proof}[Proof of the estimate \eqref{du-L2}]
\

We prove the bound \eqref{du-L2} by induction. Assume that for any $|\be|<|\al|$
\begin{equation}\label{u-ass2}
	\|\<\nab\>^{(N-|\be|)\de}\nab\Ga^\be u\|_{L^2}\lesssim  \<t\>^{-1/2}\ln^{|\be|+1}(e+t)(E^{1/2}_{|\be|+2}+\XX^{1/2}_{|\be|+2}) \,.
\end{equation}
By Duhamel's formula \eqref{Duhamel}, we have
\begin{align*}
	&\<\nab\>^{(N-|\al|)\de}\nab\Ga^{\al}u(t)\\
	&=e^{t\De}\<\nab\>^{(N-|\al|)\de}\nab\Ga^\al u(0)+\int_0^t e^{(t-s)\De}\<\nab\>^{(N-|\al|)\de}\Big[\nab\De\sum_{\be_1=0}^{\al_1-1}C_{\al_1}^{\be_1}(-1)^{\al_1-\be_1}S^{\be_1}\Ga^{\al'}u(s)\\
	&\quad -\nab\P\sum_{\be+\ga=\al}C_\al^\be[ \Ga^\be u\cdot\nab \Ga^\ga u+ \d_j\Gamma^{\be}\phi\cdot\d_j\Gamma^{\ga}\phi)]\Big]\ ds\,.
\end{align*}
Then
\begin{align*}
	\|t^{1/2}\<\nab\>^{(N-|\al|)\de} \nab\Ga^\al u\|_{\Lf}\les H_0+H_1+H_2+H_3+H_4\,,
\end{align*}
where 
\begin{align*}
	&H_0=\|t^{1/2}e^{t\De}\<\nab\>^{(N-|\al|)\de}\nab\Ga^\al u(0)\|_{L^2}\,,\\
	&H_1=\sum_{|\be|<|\al|}\|t^{1/2}\int_0^t e^{(t-s)\De}\<\nab\>^{(N-|\al|)\de}\nab\De\Ga^{\be}u(s)\ ds\|_{L^2}\,,\\
	&H_2=\sum_{\be+\ga=\al}\|t^{1/2}\int_0^{t/2} e^{(t-s)\De}   \<\nab\>^{(N-|\al|)\de}\nab\big(\Ga^\be u \cdot\nab \Ga^\ga u+\nab(\nab\Ga^\be\phi \nab\Ga^\ga\phi)\big)(s)\ ds\|_{L^2}\,,\\
	&H_3=\sum_{\be+\ga=\al}\|t^{1/2}\int_{t/2}^t e^{(t-s)\De}   \<\nab\>^{(N-|\al|)\de}\nab\P(\Ga^\be u \cdot\nab \Ga^\ga u)(s)\ ds\|_{L^2}\,,\\
	&H_4=\sum_{\be+\ga=\al}\|t^{1/2}\int_{t/2}^t e^{(t-s)\De}  \<\nab\>^{(N-|\al|)\de}\nab^2\P(\nab\Ga^\be \phi \nab\Ga^\ga\phi)(s)\ ds\|_{L^2}\,.
\end{align*}

By \eqref{Dec_heat} the first term $H_0$ is bounded by $E_{|\al|+2}^{1/2}$.
Similar to the estimates of $\KK_1$ and $\KK_{21}$ in \eqref{KK1} and \eqref{K21} respectively, by \eqref{Dec_heat} and \eqref{u-ass2} we have
\begin{align*}
	H_1&\lesssim \|\Ga^{|\al|-1} u\|_{\Lf L^2}+\ln(e+t)\|s^{1/2}\<\nab\>^{(N-|\al|+1)\de}\nab\Ga^\be u\|_{\Lf(t/2,t;L^2)}\\
	&\lesssim  \ln^{|\al|+1}(e+t)(E^{1/2}_{|\al|+2}+\XX^{1/2}_{|\al|+2})\,.
\end{align*}
and 
\begin{align*}
	H_2 &\lesssim \sum_{|\be+\ga|\leq |\al|} \|\Ga^\be u\|_{L^\infty L^2} \| \Ga^\ga u\|_{\Lf L^2} +\|\nab\Ga^\be\phi\|_{L^\infty L^2} \|\nab\Ga^\ga\phi\|_{L^\infty L^2}\lesssim E_{|\al|}\,.
\end{align*}
For the term $H_3$, by \eqref{Decay-u} we have
\begin{align*}
	H_3&\les \sum_{|\be+\ga|= |\al|}\int_{t/2}^{t} s^{1/2}(t-s)^{-1/2} \big(1+(t-s)^{-\frac{(N-|\al|)\de}{2}}\big)  \|(\Ga^\be u\cdot\nab\Ga^\ga u)(s)\|_{L^2}\ ds\\
	&\lesssim \sum_{|\be+\ga|= |\al|}\|s^{1/2}\Ga^{\be} u\|_{\Lf(t/2,t;L^\infty)} \|s^{1/2}\nab\Ga^{\ga} u\|_{\Lf(t/2,t;L^2)} \\
	&\les \sum_{|\be+\ga|= |\al|}\ln^{|\be|}(e+t)(E^{1/2}_{|\be|+1}+\XX^{1/2}_{|\be|+2}) \|s^{1/2}\nab\Ga^{\ga} u\|_{\Lf(t/2,t;L^2)}\,.
\end{align*}

Finally, by \eqref{AwayCone0} and \eqref{NearCone0} the term $H_4$ is treated as
\begin{align*}
	H_4 &\les \int_{t/2}^{t-1} (t-s)^{-1}\big(1+(t-s)^{-\frac{(N-|\al|)\de}{2}}\big) s^{1/2}\|\nab\Ga^\be \phi\nab\Ga^\ga\phi\|_{L^2}\ ds\\
	&\quad+\int_{t-1}^t (t-s)^{-1/2}\big(1+(t-s)^{-\frac{(N-|\al|)\de}{2}}\big) s^{1/2}\|\nab(\nab\Ga^\be \phi\nab\Ga^\ga\phi)\|_{L^2}\ ds\\
	&\lesssim \ln(e+t)\sum_{|\be+\ga|= |\al|}E^{1/2}_{|\be|}(E^{1/2}_{|\ga|+1}+\XX^{1/2}_{|\ga|+2})+\sum_{|\be+\ga|= |\al|}E^{1/2}_{|\be|+1}(E^{1/2}_{|\ga|+1}+\XX^{1/2}_{|\ga|+2})\\
	&\lesssim  \ln(e+t)E^{1/2}_{|\al|+1}(E^{1/2}_{|\al|+1}+\XX^{1/2}_{|\al|+2})\,.
\end{align*}
The above estimates from $H_0$ to $H_4$ yield
\begin{align*}
	&t^{1/2}\ln^{-(|\al|+1)}(e+t)\|\<\nab\>^{(N-|\al|)\de}\nab\Ga^\be u\|_{L^2}\\
	&\lesssim E^{1/2}_{|\al|+2}+\XX^{1/2}_{|\al|+2}
	+(E^{1/2}_{2}+\XX^{1/2}_2)\|s^{1/2}\ln^{-(|\al|+1)}(e+s)\nab\Ga^\al u\|_{\Lf(t/2,t;L^2)}\,.
\end{align*}
From the bootstrap assumption $E^{1/2}_N+\XX^{1/2}_{N-2}\les \ep$, the last term can be absorbed by the left hand side. Thus we obtain the bound \eqref{du-L2}.
\end{proof}

\bigskip 
Then we prove the bound \eqref{dtu-L2} using the estimates \eqref{dtu-L2al} and \eqref{du-L2}.
\begin{proof}[Proof of the estimate \eqref{dtu-L2}]
\ 

By the first bound \eqref{dtu-L2al} and \eqref{du-L2}, we obtain
\begin{align*}
	\|\d_t\Ga^\al u\|_{L^2}&\lesssim \|\nab\Ga^{ |\al|+1}u\|_{L^2}(1+E^{1/2}_{|\al|+1})+ \<t\>^{-1/2}(E^{1/2}_{[|\al|/2]+2}+\XX^{1/2}_{[|\al|/2]+3})\\
	&\les \<t\>^{-1/2}\ln^{|\al|+2}(e+t)(E^{1/2}_{|\al|+3}+\XX^{1/2}_{|\al|+3})(1+E^{1/2}_{|\al|+1})+\<t\>^{-1/2}(E^{1/2}_{[|\al|/2]+2}+\XX^{1/2}_{[|\al|/2]+3})\\
	&\les \<t\>^{-1/2}\ln^{|\al|+2}(e+t)(E^{1/2}_{|\al|+3}+\XX^{1/2}_{|\al|+3})\,.
\end{align*}
Hence, the bound \eqref{dtu-L2} follows. We complete the proof of the lemma.
\end{proof}

\bigskip 
\begin{proof}[Proof of the estimate \eqref{decay-du0}]
\ 

We assume that for any $|\be|<|\al|$
\begin{equation}    \label{ass-duLf}
	\|\<\nab\>^{(N-|\be|)\de}\nab \Ga^\be u\|_{L^\infty}\lesssim \<t\>^{-1}\ln^{|\be|+1}(e+t)(E^{1/2}_{|\be|+3}+\XX_{|\be|+3})\,.
\end{equation}

By Duhamel's formula, we have
\begin{align*}
	&\<\nab\>^{(N-|\al|)\de}\nab\Ga^{\al}u(t)\\
	&=e^{t\De}\<\nab\>^{(N-|\al|)\de}\nab\Ga^\al u(0)+\int_0^t e^{(t-s)\De}\<\nab\>^{(N-|\al|)\de}\Big[\nab\De\sum_{\be_1=0}^{\al_1-1}C_{\al_1}^{\be_1}(-1)^{\al_1-\be_1}S^{\be_1}\Ga^{\al'}u(s)\\
	&\quad -\nab\P\sum_{\be+\ga=\al}C_\al^\be[ \Ga^\be u\cdot\nab \Ga^\ga u+ \d_j(\nab\Gamma^{\be}\phi\cdot\d_j\Gamma^{\ga}\phi)]\Big]\ ds\,.
\end{align*}
Hence,
\begin{align*}
	\|t \<\nab\>^{(N-|\al|)\de} \nab \Ga^\al u\|_{\Lf}\les L_0+L_1+L_2+L_3+L_4\,,
\end{align*}
where 
\begin{align*}
	&L_0=\|t e^{t\De}\<\nab\>^{(N-|\al|)\de}\nab u(0)\|_{\Lf}\,,\\
	&L_1=\sum_{|\be|<|\al|}\|t\int_{0}^t e^{(t-s)\De} \<\nab\>^{(N-|\al|)\de}\nab  \De \Ga^\be u\ ds\|_{\Lf}\,,\\
	&L_2=\sum_{\be+\ga=\al}\|t\int_0^{t/2} e^{(t-s)\De} \<\nab\>^{(N-|\al|)\de}\nab \big(\Ga^\be u \cdot\nab\Ga^\be u+\nab(\nab\Ga^\be\phi \nab\Ga^\ga\phi)\big)(s)\ ds\|_{\Lf}\,,
\end{align*}
\begin{align*}
	&L_3=\sum_{\be+\ga=\al}\|t\int_{t/2}^t e^{(t-s)\De} \<\nab\>^{(N-|\al|)\de}\nab   \P\big(\Ga^\be u \cdot\nab\Ga^\ga  u\big)(s)\ ds\|_{\Lf}\,,\\
	&L_4=\sum_{\be+\ga=\al}\|t\int_{t/2}^t e^{(t-s)\De} \<\nab\>^{(N-|\al|)\de}\nab   \nab\P\big(\nab\Ga^\be \phi \nab\Ga^\ga\phi\big)(s)\ ds\|_{\Lf}\,.
\end{align*}

By \eqref{Dec_heat} the first term $L_0$ is bounded by $E^{1/2}_{|\al|+3}$.
In a similar way to the estimates of $\KK_1$, by \eqref{Dec_heat} and \eqref{ass-duLf} the term $L_1$ is bounded by $\ln^{|\al|+1}(e+t)(E^{1/2}_{|\al|+3}+\XX^{1/2}_{|\al|+3})$.
Similar to $\KK_{21}$, we bound the $L_2$ by $E_{|\al|}$.
For the term $L_3$, by \eqref{Dec_heat} and \eqref{Decay-u} we have
\begin{align*}
	L_3 &\lesssim \sum_{\be+\ga=\al}\int_{t/2}^{t} (t-s)^{-1/2}\big(1+(t-s)^{-\frac{(N-|\al|)\de}{2}}\big)  \|\Ga^\be u\|_{\Lf} \|s \nab \Ga^\ga  u\|_{L^\infty }\ ds\\
	&\les \|s^{1/2}\Ga^\be u\|_{\Lf(t/2,t;\Lf)}\|s \nab\Ga^\ga u\|_{\Lf(t/2,t;\Lf)}\\
	&\lesssim \sum_{|\be+\ga|=|\al|}\ln^{|\be|}(e+t) (E^{1/2}_{|\be|+2}+\XX^{1/2}_{|\be|+2})\|s \nab\Ga^\ga u\|_{\Lf(t/2,t;\Lf)}\,.
\end{align*}
Finally, similar to the estimate of $L_3$, by \eqref{AwayCone0} and \eqref{NearCone0} the term $L_3$ is treated as
\begin{align*}
	L_4&\lesssim \ln(e+t) \sum_{\be+\ga=\al} \|s^{1/2}\nab\Ga^\be \phi\|_{\Lf} \|s^{1/2}\nab^2\Ga^\ga  \phi\|_{L^\infty }\\
	&\les \ln(e+t) \sum_{|\be+\ga|=|\al|} ( E^{1/2}_{|\be|+1}+\XX^{1/2}_{|\be|+2})( E^{1/2}_{|\ga|+2}+\XX^{1/2}_{|\ga|+3})\,.
\end{align*}

Collecting the above bounds for $L_0,\cdots,\ L_4$, we obtain
\begin{align*}
	&\|t\ln^{-(|\al|+1)}(e+t)\<\nab\>^{(N-|\al|)\de}\nab \Ga^\al u\|_{\Lf L^\infty}\\
	&\les E^{1/2}_{|\al|+3}+\XX^{1/2}_{|\al|+3}+\sum_{|\be+\ga|=|\al|} ( E^{1/2}_{|\be|+2}+\XX^{1/2}_{|\be|+2})( E^{1/2}_{|\ga|+3}+\XX^{1/2}_{|\ga|+3})\\
	&\quad+(E^{1/2}_{2}+\XX^{1/2}_{2})\|s \ln^{-(|\al|+1)} \nab\Ga^\al u\|_{\Lf(t/2,t;\Lf)}\,.
\end{align*}
By $N\geq 10$ and the smallness of $E^{1/2}_N+\XX^{1/2}_{N-2}\les \ep$, the last term can be absorbed by the left hand side.
Thus the estimate \eqref{decay-du0} follows.
\end{proof}

\bigskip 
Finally, we prove the last two estimates in Proposition \ref{Dec-lem} iii). In fact, we can obtain the following general decay estimates of $\nab^{1+\si} \Ga^\al u$.
\begin{lemma}\label{dec-duLem}
	Let $1/2\leq \si\leq 3/2$. Assume that $(u,\phi)$ is the solution of \eqref{Main_Sys_VecFie} satisfying \eqref{MainAss_dini} and \eqref{Main_Prop_Ass1}.	For any $|\al|\leq N-4$ we have the improved decays
	\begin{align}\label{decay-d2u}
		&\|\nab^{1+\si}\Ga^\al u\|_{L^\infty}\lesssim \<t\>^{-1} (E^{1/2}_{|\al|+4}+\XX_{|\al|+4}^{1/2}) \,.
	\end{align}
\end{lemma}

We should remark that the bound \eqref{decay-d2u} also holds for any fixed $\si>0$. The additional derivatives $\nab^\si$ allow us to gain more decays, which guarantees the above estimates.

The bound \eqref{decay-d2Gau} is a direct consequence of the above bound with $\si=1$. Hence it suffices to prove the general estimates \eqref{decay-d2u}.

\bigskip 
\begin{proof}[Proof of the estimate \eqref{decay-d2u}]
	\ 
	
	Let $\de=1/100$. Here we would prove a more general form: for any $|\al|\leq N-4$ and integer $-(N-|\al|)\leq j_\al\leq N-|\al|$
	\begin{equation}    \label{d2u-general}
		\|\nab^{1+\si+j_\al\de} \Ga^\al u\|_{L^\infty}\lesssim \<t\>^{-1} (E^{1/2}_{|\al|+4}+\XX_{|\al|+4}^{1/2})\,.
	\end{equation}
   This estimate is proved by induction. Precisely, we assume that the following bound holds for any $|\be|<|\al|$ and $-(N-|\be|)\leq j_\be\leq N-|\be|$
	\begin{equation}    \label{ass-d2u}
		\|\nab^{1+\si+j_\be\de} \Ga^\be u\|_{L^\infty}\lesssim \<t\>^{-1} (E^{1/2}_{|\be|+4}+\XX_{|\be|+4}^{1/2})\,.
	\end{equation}

	Then by Duhamel's formula, we have
	\begin{align*}
		\nab^{1+\si+j_\al\de}\Ga^{\al}u(t)
		&=e^{t\De}\nab^{1+\si+j_\al\de}\Ga^\al u(0)+\int_0^t e^{(t-s)\De}\Big[\nab^{1+\si+j_\al\de}\De\sum_{\be_1=0}^{\al_1-1}C_{\al_1}^{\be_1}(-1)^{\al_1-\be_1}S^{\be_1}\Ga^{\al'}u(s)\\
		&\quad -\nab^{1+\si+j_\al\de}\P\sum_{\be+\ga=\al}C_\al^\be[ \Ga^\be u\cdot\nab \Ga^\ga u+ \d_j(\nab\Gamma^{\be}\phi\cdot\d_j\Gamma^{\ga}\phi)]\Big]\ ds\,.
	\end{align*}
	Using \eqref{Dec_heat}, the $\Lf$-norm of the first term is bounded by $\<t\>^{-1}E^{1/2}_{|\al|+4}$. For the lower-order linear term, by \eqref{ass-d2u} we have
	\begin{align*}
		\sum_{|\be|<|\al|}\|\int_{0}^t e^{(t-s)\De} \nab^{1+\si+j_\al\de}  \De \Ga^\be u\ ds\|_{\Lf}
		&\les \int_0^{t/2} (t-s)^{-2-\si/2-j_\al\de/2}\|\Ga^{|\al|-1}u\|_{L^2}\\
		&\quad+\int_{t/2}^{t-1} (t-s)^{-1-\de/2} \|\nab^{1+\si+(j_\al-1)\de}\Ga^\be u\|_{\Lf}ds\\
		&\quad+\int_{t-1}^{t} (t-s)^{-1+\de/2} \|\nab^{1+\si+(j_\al+1)\de}\Ga^\be u\|_{\Lf}ds\\
		&\les \<t\>^{-1} E^{1/2}_{|\al|-1}+\<t\>^{-1}(E^{1/2}_{|\al|+4}+\XX_{|\al|+4}^{1/2}) 
	\end{align*}
	Similar to $L_2$, by \eqref{Dec_heat} we easily have
	\begin{align*}
		&\sum_{\be+\ga=\al}\|\int_0^{t/2} e^{(t-s)\De}
		\nab^{1+\si+j_\al\de}\P \big(\Ga^\be u\cdot\nab \Ga^\ga u+ \d_j(\nab\Gamma^{\be}\phi\cdot\d_j\Gamma^{\ga}\phi)\big)\ ds\|_{\Lf}\\
		&\les \int_0^{t/2}(t-s)^{-2-\si/2-j_\al\de/2}(\|\Ga^\be u\cdot \Ga^\ga\|_{L^1}+\|\nab\Ga^\be\phi\nab\Ga^\ga\phi\|_{L^1})  \ ds\\
		&\lesssim \<t\>^{-1} E_{|\al|}\,.
	\end{align*}
	Next, we use \eqref{Decay-u} and \eqref{decay-du0} to bound 
	\begin{align*}
		&\sum_{\be+\ga=\al}\|\int_{t/2}^t e^{(t-s)\De}
		\nab^{1+\si+j_\al\de}\P (\Ga^\be u\cdot\nab \Ga^\ga u)\ ds\|_{\Lf}\\
		&\les \int_{t/2}^{t-1} (t-s)^{-1/2-\si/2-j_\al\de/2}\|\Ga^\be u\|_{\Lf}\|\nab\Ga^\ga u\|_{\Lf}\\
		&\quad+\int_{t-1}^t (t-s)^{-\si/2-j_\al\de/2}\|\nab(\Ga^\be u\cdot\nab\Ga^\ga u)\|_{\Lf}\\
		&\les t^{1/4-j_\al\de/2-3/2+\de/2}\sum_{|\be+\ga|=|\al|}\sum_{l=0}^1(E^{1/2}_{|\be|+2+l}+\XX^{1/2}_{|\be|+2+l})(E^{1/2}_{|\ga|+4-l}+\XX^{1/2}_{|\ga|+4-l})\\
		&\les \<t\>^{-1}(E^{1/2}_{|\al|+4}+\XX_{|\al|+4}^{1/2})\,.
	\end{align*}
In the last inequality, we have used $N\geq 10$ and \eqref{Main_Prop_Ass1}.
	Using \eqref{AwayCone0} and \eqref{NearCone0}, we bound 
	\begin{align*}
		&\sum_{\be+\ga=\al}\|\int_{t/2}^t e^{(t-s)\De}
		\nab^{1+\si+j_\al\de}\P  \d_j(\nab\Gamma^{\be}\phi\cdot\d_j\Gamma^{\ga}\phi)\ ds\|_{\Lf}\\
		&\les \int_{t/2}^{t-1}(t-s)^{-1-\si/2-j_\al\de/2}\|\nab\Ga^\be\phi\nab\Ga^\ga\phi\|_{\Lf}  \ ds\\
		&\quad+\int_{t-1}^t (t-s)^{-\si/2-j_\al\de/2}\|\nab^2(\nab\Ga^\be\phi\nab\Ga^\ga\phi)\|_{\Lf}  \ ds\\
		&\lesssim \<t\>^{-1} \sum_{|\be+\ga|=|\al|}\sum_{l=0}^2 (E^{1/2}_{|\be|+1+l}+\XX^{1/2}_{|\be|+2+l})(E^{1/2}_{|\ga|+3-l}+\XX^{1/2}_{|\ga|+4-l})\\
		&\les \<t\>^{-1}(E^{1/2}_{|\al|+3}+\XX^{1/2}_{|\al|+4}) \,.
	\end{align*}
	Collecting the above bounds, we obtain the desired estimate \eqref{d2u-general}. Hence, the bound \eqref{decay-d2u} follows.
\end{proof}

\bigskip

As a consequence of \eqref{decay-d2Gau}, we then obtain the decay estimates \eqref{decay-dtu}.

\begin{proof}[Proof of the estimate \eqref{decay-dtu}]
\ 

By the $u$-equation in \eqref{Main_Sys} and \eqref{decay-d2u}, we have
\begin{align*}
	\|\d_t \Ga^{\al} u\|_{L^\infty}\lesssim \sum_{|\be|\leq |\al|}\|\De \Ga^\be u\|_{L^\infty}+\sum_{|\be+\ga|= |\al|}\|\P(\Ga^\be u\cdot\nab \Ga^\ga u)\|_{L^\infty}+\sum_{|\be+\ga|= |\al|}\|\P\d_j(\nab\Ga^\be\phi\d_j\Ga^\ga\phi)\|_{L^\infty}\,.
\end{align*}
The first term has beed estimated in \eqref{decay-d2u}.
For the second term, by Sobolev embedding, \eqref{Decay-u}, \eqref{decay-du0} and \eqref{decay-d2u} we bound it by
\begin{align*}
	\sum_{|\be+\ga|= |\al|}\|\P(\Ga^\be u\cdot\nab \Ga^\ga u)\|_{L^\infty}&\lesssim \sum_{|\be+\ga|= |\al|}\|\Ga^\be u\cdot \nab \Ga^\ga u\|_{W^{1,4}}\\
    &\lesssim \sum_{|\be+\ga|= |\al|}\|\Ga^\be u\|_{W^{1,4}}\|\nab \Ga^\ga u\|_{L^\infty} +\|\Ga^\be u\|_{L^4}\|\nab^2 \Ga^\ga u\|_{L^\infty}\\
	&\lesssim  \sum_{|\be+\ga|= |\al|}\|\<\nab\>\Ga^\be u\|_{L^2}^{1/2}\|\<\nab\>\Ga^\be u\|_{\Lf}^{1/2}\<t\>^{-1+\de}(E^{1/2}_{|\ga|+3}+\XX^{1/2}_{|\ga|+3})\\
	&\quad+\sum_{|\be+\ga|= |\al|}E^{1/2}_{|\be|+1} \<t\>^{-1} (E^{1/2}_{|\ga|+4}+\XX_{|\ga|+4}^{1/2})\\
	&\les \<t\>^{-1}\ep (E^{1/2}_{|\al|+4}+\XX_{|\al|+4}^{1/2})\,.
\end{align*}
For the last term, by the decays \eqref{AwayCone0} and \eqref{NearCone0} we have
\begin{align*}
	\sum_{|\be+\ga|= |\al|}\|\P\d_j(\nab\Ga^\be\phi\d_j\Ga^\ga\phi)\|_{L^\infty} &\les \sum_{|\be+\ga|= |\al|}\|P_{\leq 0}(\nab\Ga^\be\phi\d_j\Ga^\ga\phi)\|_{L^\infty}+\|P_{> 0}\nab^{1+\de}(\nab\Ga^\be\phi\d_j\Ga^\ga\phi)\|_{L^\infty}\\
	&\lesssim \sum_{|\be+\ga|= |\al|}\|\nab\Ga^\be\phi\nab \Ga^\ga \phi\|_{L^\infty}+\|\nab^{1+\de}(\nab\Ga^\be\phi\nab\Ga^\ga\phi)\|_{L^\infty}\\
	&\lesssim  \sum_{|\be+\ga|= |\al|}\<t\>^{-1}\sum_{|\be+\ga|=|\al|}\sum_{l=0}^2 (E^{1/2}_{|\be|+1+l}+\XX^{1/2}_{|\be|+2+l})(E^{1/2}_{|\ga|+3-l}+\XX^{1/2}_{|\ga|+4-l})\\
	&\lesssim  \<t\>^{-1}(E^{1/2}_{|\al|+3}+\XX^{1/2}_{|\al|+4})\,.
\end{align*}
The above three estimates combined with \eqref{decay-d2u} implies the bound \eqref{decay-dtu}.
\end{proof}

\bigskip

\section{\texorpdfstring{$L^2$}{} Weighted estimates}\label{sec-L2}

The goal of this section is to bound the weighted $L^2$ generalized energies $\XX_k$, which is a bridge connecting energy bounds and decay estimates. Our proof is divided into two steps. First, we prove that the lower-order $L^2$ weighted norms $\XX^{1/2}_{N-2}$ are controlled by energy. This will be used to prove the improved estimates in \eqref{Main_Prop_result1} Then by this estimate and decay estimates in Section \ref{sec-decay}, we show an appropriate bound for the higher order $L^2$ weighted norms.

We recall the dissipative energy and ghost weight energy as 
\begin{equation*}
	\mathcal D_{N}(t)=\sum_{\al\leq N}\Big(\|\nab\Ga^\al u\|_{L^2}+\Big\|\frac{\Ga^\al u}{\<t-r\>}\Big\|_{L^2}+ \Big\|\frac{\om \d_t \Ga^\al \phi+\nab\Ga^\al \phi}{\<t-r\>}\Big\|_{L^2}\Big)\,.
\end{equation*}
The main result in the section is as follows.
\begin{proposition}\label{propXX}
	Let $N\geq 10$. Assume that $(u,\phi)$ is the solution of \eqref{Main_Sys} satisfying the assumption \eqref{Main_Prop_Ass1}, then we have the lower order bound
	\begin{align}  \label{XX_ka-2}
		\XX^{1/2}_{N-2}&\lesssim E^{1/2}_{N-2}+\ep^2\,,
	\end{align}
and the higher order bound
\begin{align}	\label{XX_ka}
	\XX^{1/2}_{N}&\lesssim  \ep+\ep \<t\>^{1/2} \DD_{N}\,.
\end{align}
\end{proposition}

\bigskip 
We begin with the Klainerman-Sideris type estimate for wave operator, which was first established in \cite{KS96}.
\begin{lemma}  \label{XX-lemma}
	There holds
	\begin{equation}    \label{Control-X}
		\mathcal X^{1/2}_1\lesssim E^{1/2}_1+\<t\>\|(\d^2_t-\De)\phi\|_{L^2}\,.
	\end{equation}
\end{lemma}

Next, we utilize the $\Ga^\al\phi$-equations in \eqref{Main_Sys_VecFie} to control the last term on the right hand side.

\begin{lemma}\label{Boxphi}
	Let $N\geq 10$. Assume that $(u,\phi)$ is a solution of \eqref{Main_Sys} satisfying \eqref{Main_Prop_Ass1}, there holds
	\begin{align}  \label{Cont-phi_ka-2}
		\<t\>\|(\d_t^2-\De)\Ga^{ N-3}\phi\|_{L^2}\lesssim \ep^2\,.
	\end{align}
and the higher-order estimates
\begin{equation}	\label{Cont-phi-ka}
	\<t\>\|(\d_t^2-\De)\Ga^{ N-1}\phi\|_{L^2}\lesssim \ep^2+ \ep \<t\>^{1/2} \DD_{N} +\ep \XX^{1/2}_N\,.
\end{equation}
\end{lemma}

Then we will prove the Proposition \ref{propXX} assuming that the Lemma \ref{Boxphi} holds.
\begin{proof}[Proof of Proposition \ref{propXX}.]
	\

	The bound \eqref{XX_ka-2} is obtained by \eqref{Control-X} and \eqref{Cont-phi_ka-2}.
	For the second bound \eqref{XX_ka}, by \eqref{Control-X} and \eqref{Cont-phi-ka} we have
	\begin{align*}
		\XX^{1/2}_{N}&\lesssim \ep+ \ep \<t\>^{1/2} \DD_{N} +\ep^2+\ep \XX^{1/2}_N\,.
	\end{align*}
	The last term can be absorbed by the left hand side. Then the desired bound \eqref{XX_ka} follows.
\end{proof}

\bigskip

Finally, we turn to the proof of Lemma \ref{Boxphi}, which completes the proof of Proposition \ref{propXX}.  

\begin{proof}[Proof of the bound \eqref{Cont-phi_ka-2}]
	
	\ 
	
	\medskip 
	By the $\Ga^\al \phi$-equations in \eqref{Main_Sys_VecFie}, for any $|\al|\leq N-3$ we have
	\begin{align*}
		\<t\>\|(\d^2_t-\De)\Ga^\al \phi\|_{L^2}&\les \sum_{\be+\ga=\al} \<t\>\|\Ga^\be u\cdot\nab\d_t\Ga^\ga\phi\|_{L^2}+\sum_{\be+\ga=\al} \<t\>\|\d_t\Ga^\be u\cdot\nab\Ga^\ga\phi\|_{L^2}\\
		&\quad+\sum_{\be+\ga+\de=\al} \<t\>\|\Ga^\be u\cdot\nab(\Ga^\ga u\cdot\nab \Ga^\de \phi)\|_{L^2}\\
		&=B_1+B_2+B_3\,.
	\end{align*}

\medskip 
\emph{Step 1: We claim that for any $|\al|\leq N-3$}
    \begin{align*}
    	\<t\>\|(\d_t^2-\De)\Ga^{ \al}\phi\|_{L^2}\lesssim \ep^2+\ep \XX^{1/2}_{N-2}+\ep\XX_{N-2}\,.
    \end{align*}
    
\bigskip 

\emph{i) Estimates of $B_1$.} 

\medskip 

When $r<\frac{2}{3}\<t\>$, by \eqref{XX} and \eqref{AwayCone} we have
\begin{align*}
	&\sum_{|\be+\ga|\leq N-3} \<t\>\|\Ga^\be u\cdot\nab\d_t\Ga^\ga\phi\|_{L^2(r\leq 2\<t\>/3)}\\
	&\lesssim  \|\Ga^{[(N-3)/2]}u\|_{L^\infty}\|\<t-r\>\nab\d_t\Ga^{N-3} \phi\|_{L^2(r\leq 2\<t\>/3)}+\|\Ga^{N-3}u\|_{L^2}\|t\nab\d_t\Ga^{[(N-3)/2]}\phi\|_{L^\infty(r\leq 2\<t\>/3)}\\
	&\lesssim E^{1/2}_{[(N-3)/2]+2}\XX^{1/2}_{N-2}+E^{1/2}_{N-2}\XX^{1/2}_{[(N-3)/2]+3}\\
	&\les E^{1/2}_{N-2}\XX^{1/2}_{N-2}\,.
\end{align*}
When $r\geq \frac{2}{3}\<t\>$, by \eqref{omf} and we split the integrand into
\begin{align*}
	&\sum_{|\be+\ga|\leq N-3} \<t\>\|\Ga^\be u\cdot\nab\d_t\Ga^\ga\phi\|_{L^2(r\geq 2\<t\>/3)}\\
	&=\sum_{|\be+\ga|\leq N-3} \|\<t\>\Ga^\be u\cdot(\om\d_r\d_t\Ga^\ga\phi+\frac{\om^\bot}{r}\Om\d_t\Ga^\ga \phi)\|_{L^2(r\geq 2\<t\>/3)}\\
	&\les \sum_{|\be+\ga|\leq N-3} (\|r\Ga^\be u\cdot\om\d_r\d_t\Ga^\ga\phi\|_{L^2(r\geq 2\<t\>/3)}+\|\Ga^\be u \d_t\Om\Ga^\ga \phi\|_{L^2(r\geq 2\<t\>/3)})\\
	&\les E^{1/2}_{N-1}E^{1/2}_{N-2}+E^{1/2}_{N-3}E^{1/2}_{N-2}\\
	&\les E^{1/2}_{N-1}E^{1/2}_{N-2}\,.
\end{align*}

Since $N\geq 10$ and $|\al|\leq N-3$ we obtain
\begin{align*}
	B_1\lesssim E^{1/2}_{N-2}\XX^{1/2}_{N-2}+E^{1/2}_{N-1}E^{1/2}_{N-2}\,.
\end{align*}

\medskip 

\emph{ii) Estimates of $B_2$.} 

\medskip 

When $r\geq \frac{2}{3}\<t\>$, we write $B_2$ as 
\begin{align*}
	B_2&=\sum_{\be+\ga=\al}\<t\>\big\|\d_t \Ga^\be u\cdot (\om\d_r \Ga^\ga\phi+\frac{\om^\bot}{r}\Om \Ga^\ga\phi)\big\|_{L^2}\\
	&=\sum_{\be+\ga=\al} \big\|\frac{\<t\>}{r}(r\om\cdot \d_t \Ga^\be u \d_r \Ga^\ga\phi)+\frac{\<t\>\om^\bot}{r}\d_t \Ga^\be u\Om \Ga^\ga\phi\big\|_{L^2}\,.
\end{align*}
Then by \eqref{omf} and  \eqref{phi-lf} we have
\begin{align*}
	&\sum_{\be+\ga=\al} \<t\>\|\d_t\Ga^\be u\cdot\nab\Ga^\ga\phi\|_{L^2(r\geq 2\<t\>/3)}\\
	&\les \sum_{\be+\ga=\al} \|(r\om\cdot\d_t\Ga^\be u) \d_r\Ga^\ga\phi\|_{L^2}+\|\d_t \Ga^\be u \Om \Ga^\ga\phi\|_{L^2}\\
	&\les \|r\om\cdot\d_t\Ga^{N-3} u\|_{\Lf} \|\d_r\Ga^\ga\phi\|_{L^2}+\|\d_t \Ga^\be u\|_{L^2} \|\Om \Ga^\ga\phi\|_{\Lf}\\ 
	&\les E^{1/2}_N E^{1/2}_{N-2}+E^{1/2}_{N-2}E^{1/2}_{N-1}\\
	&\les E^{1/2}_N E^{1/2}_{N-2}\,.
\end{align*}

When $r<\frac{2}{3}\<t\>$, by \eqref{decay-dtu}, \eqref{dtu-L2} and \eqref{AwayCone0} we have
\begin{align*}
	B_2&\lesssim \|t\d_t u\|_{\Lf}\|\nab\Ga^{N-3}\phi\|_{L^2}+\|\d_t\Ga^{N-3} u\|_{L^2}\|t\nab\Ga^{N-4}\phi\|_{\Lf}\\
	&\lesssim (E^{1/2}_{4}+\XX_{4}^{1/2}+\XX_{4}^{3/2})E^{1/2}_{N-3}+ \<t\>^{-1/2+\de}(E^{1/2}_{N}+\XX^{1/2}_{[(N-3)/2]+4})(E^{1/2}_{N-4}+\XX^{1/2}_{N-2})\\
	&\lesssim \ep(\ep+\XX^{1/2}_{N-2}+\XX^{3/2}_{N-2})+(\ep+\XX^{1/2}_{N-2})(\ep+\XX^{1/2}_{N-2})\\
	&\les \ep^2+\ep \XX^{1/2}_{N-2}+\ep\XX_{N-2}^{3/2}\,.
\end{align*}

\medskip 

\emph{iii) Estimates of $B_3$.}

\medskip 

The term $B_3$ is rewritten as 
\begin{align*}
	B_3=\sum_{\be+\ga+\de=\al} \<t\>\|\Ga^\be u\cdot(\Ga^\ga u\cdot\nab\nab \Ga^\de \phi)\|_{L^2}+\sum_{\be+\ga+\de=\al} \<t\>\|(\Ga^\be u\cdot\nab\Ga^\ga u)\cdot\nab \Ga^\de \phi\|_{L^2}=:B_{31}+B_{32}\,.
\end{align*}

We bound the term $B_{31}$ first. By symmetry we assume $|\be|\leq [(N-3)/2]$. Then in the region $r<2\<t\>/3$, similar to the \emph{i) Estimates of $B_1$} we have
\begin{align*}
	&\sum_{|\be+\ga+\de|\leq N-3;|\be|\leq  [(N-3)/2]} \<t\>\|\Ga^\be u\cdot\Ga^\ga u\cdot\nab\nab \Ga^\de \phi\|_{L^2(r<2\<t\>/3)}
	\les \|\Ga^{ [(N-3)/2]}u\|_{\Lf} E^{1/2}_{N-2}\XX^{1/2}_{N-2}\les E_{N-2}\XX^{1/2}_{N-2}\,.
\end{align*}
In the region $r\geq 2\<t\>/3$, we have
\begin{align*}
	&\sum_{|\be+\ga+\de|\leq N-3;|\be|\leq  [(N-3)/2]} \<t\>\|\Ga^\be u\cdot\Ga^\ga u\cdot\nab\nab \Ga^\de \phi\|_{L^2(r\geq 2\<t\>/3)}\\
	&=\sum_{|\be+\ga+\de|\leq N-3;|\be|\leq  [(N-3)/2]} \<t\>\big\|\Ga^\ga u\Ga^\be u\cdot(\om\d_r+\frac{\om^\bot}{r}\Om)\nab \Ga^\de \phi\big\|_{L^2(r\geq 2\<t\>/3)}\\
	&\les \|\Ga^{N-3}u\|_{L^2}\big\|r\Ga^{[(N-3)/2]} u\cdot(\om\d_r+\frac{\om^\bot}{r}\Om)\nab \Ga^{[(N-3)/2]} \phi\big\|_{\Lf(r\geq 2\<t\>/3)}\\
	&\quad+\|\Ga^{[(N-3)/2]}u\|_{\Lf}\big\|r\Ga^{[(N-3)/2]} u\cdot(\om\d_r+\frac{\om^\bot}{r}\Om)\nab \Ga^{N-3} \phi\big\|_{L^2(r\geq 2\<t\>/3)}\,.
\end{align*}
We then use \eqref{omf} and Sobolev embeddings to bound this by
\begin{align*}
	&E^{1/2}_{N-3}(\|r\om\cdot \Ga^{[(N-3)/2]}u\|_{\Lf}+E^{1/2}_{[(N-3)/2]+2})E^{1/2}_{[(N-3)/2]+3}\\
	&+E^{1/2}_{[(N-3)/2]+2}(\|r\om\cdot \Ga^{[(N-3)/2]}u\|_{\Lf}+E^{1/2}_{[(N-3)/2]+2})E^{1/2}_{N-2}
	\les E_{N-2}^{3/2}\,.
\end{align*}

For the second term $B_{32}$, when $|\be|\geq [(N-3)/2]$, by \eqref{Decay-u} and \eqref{AwayCone0} we have
\begin{align*}
	&\<t\>\|\Ga^{N-3} u\cdot\nab \Ga^{[(N-3)/2]} u\cdot\nab \Ga^{[(N-3)/2]}\phi\|_{L^2(r<2\<t\>/3)}\\
	&\lesssim \|\Ga^{N-3} u\|_{L^2}\|\nab\Ga^{[(N-3)/2]} u\|_{\Lf}\|\<t\>\nab\Ga^{[(N-3)/2]}\phi\|_{\Lf(r< \frac{2}{3}\<t\> )}\\
	&\lesssim E^{1/2}_{N-3} \<t\>^{-1/2+\de}(E^{1/2}_{[(N-3)/2]+3}+\XX^{1/2}_{[(N-3)/2]+3})  \ln^{1/2}(e+t)(E^{1/2}_{[(N-3)/2]}+\XX^{1/2}_{[(N-3)/2]+2})\\
	&\les E_{N-2}(E^{1/2}_{N-2}+\XX^{1/2}_{N-2})(E^{1/2}_{N-2}+\XX^{1/2}_{N-2})\,.
\end{align*}
and by \eqref{rv}
\begin{align*}
	&\<t\>\|(\Ga^{N-3} u\cdot\nab \Ga^{[(N-3)/2]} u)\cdot\nab \Ga^{[(N-3)/2]}\phi\|_{L^2(r\geq 2\<t\>/3)}\\
	&\lesssim \|\Ga^{N-3} u\|_{L^2}\|r \om_j \nab \Ga^{[(N-3)/2]} u_j \|_{\Lf}\|\d_r\Ga^{[(N-3)/2]}\phi\|_{\Lf}\\
	&\quad+\|\Ga^{N-3} u\|_{L^2}\|\nab \Ga^{[(N-3)/2]} u\|_{\Lf}\|\Om\Ga^{[(N-3)/2]}\phi\|_{\Lf}\\
	&\les E^{3/2}_{N-2}\,.
\end{align*}
When $|\be|<[(N-3)/2]$ and $|\de|\leq |\al|-1\leq N-4$, by \eqref{Decay-u} and \eqref{AwayCone0} we have
\begin{align*}
	&\sum_{|\ga+\de|\leq N-3}\<t\>\|\Ga^{[(N-3)/2]} u\cdot\nab \Ga^\ga u\cdot\nab\Ga^\de\phi\|_{L^2(r< 2\<t\>/3)}\\
	&\lesssim \|\Ga^{[(N-3)/2]} u\|_{L^\infty }\|\nab\Ga^{N-3} u\|_{L^2}\|\<t\>\nab\Ga^{N-4}\phi\|_{\Lf(r<\frac{2}{3}\<t\>)}\\
	&\les\<t\>^{-1/2+\de}(E^{1/2}_{N-2}+\XX^{1/2}_{N-2})E^{1/2}_{N-2} (E^{1/2}_{N-4}+\XX^{1/2}_{N-2})\ln^{1/2}(e+t)\\
	&\lesssim E_{N-2} (E^{1/2}_{N-2}+\XX^{1/2}_{N-2}+E^{1/2}_{N-2}\XX^{1/2}_{N-2}+\XX_{N-2})
\end{align*}
and by \eqref{rv} we have
\begin{align*}
	&\sum_{|\ga+\de|\leq N-3}\<t\>\|\Ga^{[(N-3)/2]} u\cdot\nab \Ga^\ga u\cdot\nab\Ga^\de\phi\|_{L^2(r\geq 2\<t\>/3)}\\
	&\lesssim \|r^{1/2}\Ga^{[(N-3)/2]} u\|_{\Lf}(\|\nab \Ga^{N-3} u\|_{L^2}\|r^{1/2}\nab\Ga^{[(N-3)/2]}\phi\|_{\Lf}+\|r^{1/2}\nab \Ga^{[(N-3)/2]} u\|_{\Lf}\|\nab\Ga^{N-3}\phi\|_{L^2})\\
	&\lesssim E^{1/2}_{[(N-3)/2]+2}  E^{1/2}_{N-2}E^{1/2}_{[(N-3)/2]+3}\les E^{3/2}_{N-2}\,.
\end{align*}
Finally, when $|\de|=|\al|=N-3$, we have
\begin{align*}
	\|tu\cdot\nab u\cdot\nab\Ga^{N-3}\phi\|_{L^2}&\lesssim \|t^{1/2}u\|_{\Lf}\|t^{1/2}\nab u\|_{\Lf}\|\nab\Ga^\al\phi\|_{L^2}\\
	&\lesssim E_{N-2}(1+\XX_{N-2})E^{1/2}_{N-2}
\end{align*}

Therefore, 
\begin{align*}
	\<t\>\|(\d^2_t-\De)\Ga^{N-3}\phi\|_{L^2}&\les \ep \XX^{1/2}_{N-2}+\ep^2+\ep \XX^{3/2}_{N-2}+\ep^3\les \ep^2+\ep \XX^{1/2}_{N-2}+\ep\XX_{N-2}^{3/2}\,.
\end{align*}
The bound \eqref{Cont-phi_ka-2} is obtained.
\end{proof}	
	
\bigskip 
\begin{proof}[Proof of the bound \eqref{Cont-phi-ka}]
	
	\
	
	By the $\Ga^\al \phi$-equations in \eqref{Main_Sys_VecFie}, for any $|\al|\leq N-1$ we have
	\begin{align*}
		\<t\>\|(\d^2_t-\De)\Ga^\al \phi\|_{L^2}&\leq \sum_{\be+\ga=\al} \<t\>\|\Ga^\be u\cdot\nab\d_t\Ga^\ga\phi\|_{L^2}+\sum_{\be+\ga=\al} \<t\>\|\d_t\Ga^\be u\cdot\nab\Ga^\ga\phi\|_{L^2}\\
		&\quad +\sum_{\be+\ga+\de=\al} \<t\>\|\Ga^\be u\cdot\nab(\Ga^\ga u\cdot\nab \Ga^\de \phi)\|_{L^2}\\
		&=\BB_1+\BB_2+\BB_3\,.
	\end{align*}
    
    \bigskip
    
    \emph{i) Estimates of $\BB_1$.} 
    
    \medskip 
    When $r<\frac{2}{3}\<t\>$, by \eqref{XX} and \eqref{AwayCone} we have
    \begin{align*}
    	&\sum_{|\be+\ga|\leq N-1} \<t\>\|\Ga^\be u\cdot\nab\d_t\Ga^\ga\phi\|_{L^2(r<\frac{2}{3}\<t\>)}\\
    	&\lesssim \|\Ga^{ [(N-1)/2]}u\|_{L^\infty}\|\<t-r\>\nab\d_t\Ga^{N-1} \phi\|_{L^2(r< 2\<t\>/3)}+\|\Ga^{N-1}u\|_{L^2}\|t\nab\d_t\Ga^{[(N-1)/2]}\phi\|_{L^\infty(r< 2\<t\>/3)}\\
    	&\lesssim E^{1/2}_{[(N-1)/2]+2}\XX^{1/2}_{N}+E^{1/2}_{N-1}\XX^{1/2}_{[(N-1)/2]+3}\,.
    \end{align*}
    When $r\geq \frac{2}{3}\<t\>$, we split the $B_1$ into
    \begin{align*}
    	&\sum_{|\be+\ga|\leq N-1} \<t\>\|\Ga^\be u\cdot\nab\d_t\Ga^\ga\phi\|_{L^2(r\geq \frac{2}{3}\<t\>)}\\
    	&=\sum_{|\be+\ga|\leq N-1}\<t\>\|\Ga^\be u\cdot \om\d_r\d_t\Ga^\ga\phi\|_{L^2(r\geq 2\<t\>/3)}+\sum_{|\be+\ga|\leq N-1}\<t\>\|\Ga^\be u\cdot \frac{\om^{\bot}}{r}\Om\d_t\Ga^\ga\phi\|_{L^2(r\geq 2\<t\>/3)}\\
    	:&=\BB_{11}+\BB_{12}\,.
    \end{align*}
For the fist term $\BB_{11}$, we use \eqref{omf}, \eqref{AwayCone} and \eqref{NearCone} to gain
\begin{align*}
	\BB_{11}&\lesssim \|r\om\cdot\Ga^{[(N-1)/2]} u\|_{L^\infty}\| \d_r\d_t\Ga^{N-1}\phi\|_{L^2}+\|\Ga^{N-1} u\|_{L^2}\|t\d_r\d_t\Ga^{[(N-1)/2]}\phi\|_{L^\infty(r\geq 5\<t\>/4)}\\
	&\quad+\<t\>^{1/2}\Big\|\frac{\om\cdot\Ga^{N-1} u+\nab\cdot \Ga^{N-1} u}{\<t-r\>}\Big\|_{L^2}\|r^{1/2}\<t-r\>\d_r\d_t\Ga^{[(N-1)/2]}\phi\|_{L^\infty(2\<t\>/3\leq r<5\<t\>/4)}\\
	&\lesssim E^{1/2}_{[(N-1)/2]+2}E^{1/2}_{N}+E^{1/2}_{N-1}\XX^{1/2}_{[(N-1)/2]+3}+\<t\>^{1/2}\DD_{N-1}(E^{1/2}_{[(N-1)/2]+2}+\XX^{1/2}_{[(N-1)/2]+3})\,.
\end{align*}
    By Sobolev embedding, the second term $\BB_{12}$ is controlled by
    \begin{align*}
    	\BB_{12}\lesssim  \|\Ga^{N-1}u\|_{L^4}\|  \d_t\Om\Ga^{[(N-1)/2]}\phi\|_{L^4}+ \|\Ga^{[(N-1)/2]}u\|_{\Lf}\|  \d_t\Om\Ga^{N-1}\phi\|_{L^2}
    	\lesssim E^{1/2}_{N}E^{1/2}_{[(N-1)/2]+2}\,.
    \end{align*}

Thus, collecting the above bounds, by $N\geq 10$ and \eqref{XX_ka-2} we have
\begin{align*}
	\BB_1&\lesssim E^{1/2}_{N-2}\XX^{1/2}_{N}+E^{1/2}_{N-1}\XX^{1/2}_{N-2}+ E^{1/2}_{N-2}E^{1/2}_{N}+\<t\>^{1/2}\DD_{N-1}(E^{1/2}_{N-2}+\XX^{1/2}_{N-2})\\ 
	&\les \ep \XX^{1/2}_N+\ep^2+\ep \<t\>^{1/2}\DD_{N-1}\,.
\end{align*}

\bigskip 

    \emph{ii) Estimates of $\BB_2$.}
     
     \medskip  
    When $\ga=\al$, by \eqref{decay-dtu} we have
    \begin{align*}
    \BB_2\mathbf{1}_{\ga=\al}\lesssim \<t\>\|\d_t u\|_{\Lf}\|\nab\Ga^{N-1}\phi\|_{L^2}\lesssim \ep E^{1/2}_{N-1}\les \ep^2\,.
    \end{align*}
    Thus it remains to consider the cases $|\ga|\leq |\al|-1\leq N-2$.
    
    In the region $r\leq 2\<t\>/3$ or $r\geq 5\<t\>/4$, by \eqref{dtu-L2al}, \eqref{dtu-L2} and \eqref{AwayCone0} we have
    \begin{align*}
    	&\sum_{|\be+\ga|\leq N-1;|\ga|\leq N-2}\|\d_t\Ga^\be u\cdot t\nab\Ga^\ga\phi\|_{L^2(r\notin[\frac{2}{3}\<t\>,\frac{5}{4}\<t\>])}\\
    	&\les \|\d_t \Ga^{N-1}u\|_{L^2}\|t\nab\Ga^{[(N-1)/2]}\phi\|_{\Lf(r\notin[\frac{2}{3}\<t\>,\frac{5}{4}\<t\>])}
    	+ \|\d_t \Ga^{[(N-1)/2]}u\|_{L^2}\|t\nab\Ga^{N-2}\phi\|_{\Lf(r\notin[\frac{2}{3}\<t\>,\frac{5}{4}\<t\>])}\\
    	&\les \big(\|\nab\Ga^{N}u\|_{L^2}(1+E^{1/2}_{N})+ \<t\>^{-1/2}(E^{1/2}_{N}+\XX^{1/2}_{N-2})\big) \ln^{1/2}(e+t)(E^{1/2}_{[(N-1)/2]}+\XX^{1/2}_{[(N-1)/2]+2})\\
    	&\quad +\<t\>^{-1/2}\ln^{[(N-1)/2]+2}(e+t)(E^{1/2}_{[(N-1)/2]+3}+\XX^{1/2}_{[(N-1)/2]+4})\ln^{1/2}(e+t)(E^{1/2}_{N-2}+\XX^{1/2}_{N})
    \end{align*}
This is bounded by
\begin{align*}
	&(\|\nab\Ga^N u\|_{L^2} +\ep\<t\>^{-1/2})\ep\<t\>^\de+\ep \<t\>^{-1/2+\de}(\ep+\XX^{1/2}_N) \\
	&\les \ep \<t\>^\de \|\nab\Ga^N u\|_{L^2} +\ep^2+\ep \XX^{1/2}_N\,.
\end{align*}

    In the region $ 2\<t\>/3<r<5\<t\>/4$, when $|\ga|\leq [(N-1)/2]$ by \eqref{dtu-L2al} and \eqref{NearCone0} we have
    \begin{align*}
    	&\|t\d_t\Ga^{N-1} u\cdot\nab\Ga^{[(N-1)/2]}\phi\|_{L^2(r\in[\frac{2}{3}\<t\>,\frac{5}{4}\<t\>])}\\
    	&\lesssim t^{1/2}\|\d_t\Ga^{N-1}u\|_{L^2}\|r^{1/2}\nab\Ga^{[(N-1)/2]}\phi\|_{\Lf(r\in[\frac{2}{3}\<t\>,\frac{5}{4}\<t\>])}\\
    	&\lesssim \<t\>^{1/2}\big(\|\nab\Ga^{N}u\|_{L^2}(1+E^{1/2}_{N})+ \<t\>^{-1/2}(E^{1/2}_{N}+\XX^{1/2}_{[(N-1)/2]+3})\big)(E^{1/2}_{[(N-1)/2]+1}+\XX^{1/2}_{[(N-1)/2]+2})\\
    	&\lesssim \ep \<t\>^{1/2}\|\nab \Ga^{ N}u\|_{L^2}+\ep^2\,.
    \end{align*}
    When $[(N-1)/2]<|\ga|\leq N-2$, by \eqref{omf} and \eqref{phi-lf} we have
    \begin{align*}
    	&\|t\d_t\Ga^{[(N-1)/2]} u\cdot \nab\Ga^{N-2}\phi\|_{L^2(r\in[\frac{2}{3}\<t\>,\frac{5}{4}\<t\>])}\\
    	&\lesssim \big\|r(\d_t\Ga^{[(N-1)/2]} u\cdot \om\d_r\Ga^{N-2}\phi+\d_t\Ga^{[(N-1)/2]} u\cdot \frac{\om^{\bot}}{r}\Om\Ga^{N-2}\phi)\big\|_{L^2(r\in[\frac{2}{3}\<t\>,\frac{5}{4}\<t\>])}\\
    	&\lesssim \|r\d_t\Ga^{[(N-1)/2]} u\cdot\om\|_{\Lf}\|\d_r\Ga^{N-2}\phi\|_{L^2}+\|\d_t\Ga^{[(N-1)/2]} u\|_{L^2}\|\Om\Ga^{N-2}\phi\|_{\Lf}\\
    	&\lesssim E^{1/2}_{[(N-1)/2]+3}E^{1/2}_{N-2}+E^{1/2}_{[(N-1)/2]+1}E^{1/2}_{N}\\
    	&\les \ep^2\,.
    \end{align*}
    
Hence, we obtain     
    \begin{align*}
    	\BB_2\lesssim \ep \<t\>^{1/2} \|\nab\Ga^N\|_{L^2} +\ep^2+\ep \XX^{1/2}_N\,.
    \end{align*}

\bigskip 

\emph{iii) Estimates of $\BB_3$.}

\medskip 
The term $\BB_3$ is rewritten as 
\begin{align*}
	\BB_3=\sum_{\be+\ga+\de=\al} \<t\>\|\Ga^\be u\cdot\Ga^\ga u\cdot\nab\nab \Ga^\de \phi\|_{L^2}+\sum_{\be+\ga+\de=\al} \<t\>\|\Ga^\be u\cdot\nab\Ga^\ga u\cdot\nab \Ga^\de \phi\|_{L^2}=:\BB_{31}+\BB_{32}\,.
\end{align*}

In the region $r\geq 2\<t\>/3$, by \eqref{rv} we have
\begin{align*}
	&\sum_{|\be+\ga+\de|\leq N-1} \<t\>\|\Ga^\be u\cdot\Ga^\ga u\cdot\nab\nab \Ga^\de \phi\|_{L^2(r\geq 2\<t\>/3)}\\
	&\les \|r^{1/2}\Ga^{[(N-1)/2]}u\|_{\Lf}^2E^{1/2}_N+E^{1/2}_{N-1}\|r^{1/2}\Ga^{[(N-1)/2]}u\|_{\Lf}\|r^{1/2}\nab^2 \Ga^{[(N-1)/2]}\phi\|_{\Lf}\\
	&\les E_{[(N-1)/2]+2}E^{1/2}_N+E^{1/2}_{N-1}E^{1/2}_{[(N-1)/2]+2}E^{1/2}_{[(N-1)/2]+3}\\
	&\les E_{N-2}E^{1/2}_N\,.
\end{align*}
In the same way, we also have
\begin{align*}
	\sum_{|\be+\ga+\de|\leq N-1} \<t\>\|\Ga^\be u\cdot\nab\Ga^\ga u\cdot\nab \Ga^\de \phi\|_{L^2(r\geq 2\<t\>/3)}
	\les E_{N-2}E^{1/2}_N\,.
\end{align*}

In the region $r<2\<t\>/3$, by Sobolev embeddings,  \eqref{AwayCone} and \eqref{XX_ka-2}, the first term can be bounded by 
\begin{align*}
	&\sum_{|\be+\ga+\de|\leq N-1} \<t\>\|\Ga^\be u\cdot\Ga^\ga u\cdot\nab\nab \Ga^\de \phi\|_{L^2(r<2\<t\>/3)}\\
	&\les \sum_{|\be+\ga+\de|\leq N-1} \|\Ga^\be u\cdot\Ga^\ga u\cdot \<t-r\>\nab\nab \Ga^\de \phi\|_{L^2(r<2\<t\>/3)}\\
	&\les \|\Ga^{[(N-1)/2]}u\|_{\Lf}^2\|\<t-r\>\nab\nab \Ga^{N-1} \phi\|_{L^2(r<2\<t\>/3)}\\
	&\quad+\|\Ga^{N-1}u\|_{L^2}\|\Ga^{[(N-1)/2]}u\|_{\Lf}\|\<t\>\nab\nab \Ga^{[(N-1)/2]} \phi\|_{\Lf(r<2\<t\>/3)}\\
	&\les E_{N-2}\XX^{1/2}_N+E^{1/2}_{N-1}E^{1/2}_{N-2}\XX^{1/2}_{[(N-1)/2]+3}\\
	&\les \ep^2 \XX^{1/2}_N+\ep^3\,.
\end{align*}
For the second term $\BB_{32}$, when $|\nu|\leq N-2$, by \eqref{Decay-u}, \eqref{AwayCone0} and \eqref{XX_ka-2} we have
\begin{align*}
	&\sum_{|\be+\ga+\nu|\leq N-1;|\nu|\leq N-2} \<t\>\|\Ga^\be u\cdot\nab\Ga^\ga u\cdot\nab \Ga^\nu \phi\|_{L^2(r<2\<t\>/3)}\\
	&\les \|\Ga^{N-1}u\|_{L^2}\|\Ga^{[(N-1)/2]}u\|_{\Lf}\|\<t\>\nab\Ga^{N-2}\phi\|_{\Lf(r<2\<t\>/3)}\\
	&\les E^{1/2}_{N-1}\<t\>^{-1/2+\de}(E^{1/2}_{[(N-1)/2]+2}+\XX^{1/2}_{[(N-1)/2]+3})(E^{1/2}_{N-2}+\XX^{1/2}_{N})\ln^{1/2}(e+t)\\
	&\les \ep^2 (\ep+\XX^{1/2}_N)\,.
\end{align*}
When $|\nu|=N-1$, by \eqref{Decay-u} and \eqref{decay-du0} we have
\begin{align*}
	& \<t\>\| u\cdot\nab u\cdot\nab \Ga^{N-1} \phi\|_{L^2(r<2\<t\>/3)}\\
	&\les \<t\>\|u\|_{\Lf}\|\nab u\|_{\Lf}E^{1/2}_{N-1}\\
	&\les \<t\>^{1-1/2+\de-1+\de}(E^{1/2}_3+\XX^{1/2}_3)(E^{1/2}_{3}+\XX_{3})E^{1/2}_{N-1} \\
	&\les \ep^3\,.
\end{align*}

Hence, the term $\BB_3$ is bounded by 
\begin{align*}
	\BB_3\les \ep^3+\ep^2 \XX^{1/2}_N\,.
\end{align*}

Collecting the above bounds, we obtain the estimate \eqref{Cont-phi-ka}.
\end{proof}

\bigskip

\section{Energy estimates}\label{sec-Energy}
This section is devoted to the energy estimates of system \eqref{Main_Sys_VecFie}. More precisely, we aim to establish uniform control over the modified energy of the solutions $(u,\phi)$, which is equivalent with the Klainerman's generalized energy.
The keys to this are to characterize these norms using energy functionals constructed with ghost weight, and to introduce a new ``good unknown", the velocity $u$.

Let $N\geq 10$ and $|\al|\leq N$. Define the ghost weight as
\begin{equation*}%   \label{GhostWgt}
	q(\si)=\arctan \si\,,\quad \si=t-r\,.
\end{equation*}
We will use the notation
\[   \<\si\>=\sqrt{1+\si^2}\,.   \]
Then we define the modified energies of system \eqref{Main_Sys_VecFie} as 
\begin{align*}
	\EE_\al(t)&=\EE_\al(u,t)+\EE_\al(\phi,t)
	:=\int_{\R^2} \frac{1}{2} e^{-q} |\Ga^\al u|^2\ dx+\int_{\R^2} \frac{1}{2} e^{-q} |\nab_{t,x}\Ga^\al \phi|^2\ dx\,.
\end{align*}
In order for the energy estimates of $\phi$, here we also need the following auxiliary norms
\begin{align*}
	\tilde\EE_\al(\phi,t)=\int_{\R^2} e^{-q}\d_t\Ga^\al\phi\Ga^\al u\cdot\nab\phi\ dx+\frac{1}{2}\int_{\R^2} e^{-q} |\Ga^\al u\cdot \nab\phi|^2 \ dx+\frac{1}{2}\int_{\R^2} e^{-q}|u\cdot\nab\Ga^\al\phi|^2\ dx\,.
\end{align*}
In view of the viscosity in the equation of $u$ and the above modified energy, we define the modified functional with respect to $\DD_N$ as
\begin{align}      \label{Dal}
	\mathscr D_\al^2(u,t)&=\int_{\R^2} \frac{1}{2}e^{-q}|\nab\Ga^\al u|^2\ dx+\int_{\R^2} \frac{1}{2}e^{-q}\frac{|\Ga^\al u|^2}{1+\si^2}\ dx\,,\\\nonumber
	\mathscr D_\al^2(\phi,t)&=\int_{\R^2} \frac{1}{2}e^{-q}\Big(\frac{|\om\d_t\Ga^\al \phi+\nab\Ga^\al\phi|^2}{1+\si^2}\Big)\ dx\,.
\end{align}

We will use the following abbreviations: for any integer $m\leq N$
\begin{align*} 
\EE_m&=\EE_m(u,t)+\EE_m(\phi,t):=\sum_{|\al|\leq m} \EE_{\al}(u,t)+\sum_{|\al|\leq m} \EE_{\al}(\phi,t)\,, \\
\mathscr D_m^2&=\mathscr D_m^2(u,t)+\mathscr D_m^2(\phi,t):=\sum_{|\al|\leq m}\mathscr D_\al^2(u,t)  +\sum_{|\al|\leq m}\mathscr D_\al^2(\phi,t) \,.      
\end{align*}

Then the energy estimates are as follows.
\begin{proposition}     \label{Prop_Eesti}
	Let $N\geq 10$ and $|\al|\leq N$. Assume that $(u,\phi)$ is the solution of \eqref{Main_Sys} satisfying \eqref{MainAss_dini} and \eqref{Main_Prop_Ass1}. Then for any $t\in[0,T]$, we have
	
	i) the equivalence relation:
	\begin{align}   \label{eq}
		\EE_\al(t)+\tilde{\EE}_\al(\phi,t)+\int_0^t \int_{\R^2} e^{-q} |\nab\Ga^\al u|^2\ dxds\approx E_\al(t)\,,\\ \label{eq-D}
        \mathscr D_m(t)\approx \DD_m(t).
	\end{align}
    
    ii) energy estimates:
	\begin{align}       \label{E-u}
		\frac{d}{dt}\EE_\al(u,t)+\frac{1}{8}\mathscr D_\al^2(u,t)\les \mathscr D_{|\al|-1}^2(u,t)+ \ep^3 \<t\>^{-1},\\ \label{E-phi}
		\frac{d}{dt}\big(\EE_\al(\phi,t)+\tilde \EE_\al(\phi,t)\big)+\mathscr D^2_\al(\phi,t)\les \ep \mathscr D_{N}^2(t)+ \ep^3\<t\>^{-1}\,.
	\end{align}
\end{proposition}

From the definitions of energy and modified energy, we can prove the equivalence first.
\begin{proof}[Proof of equivalence \eqref{eq} and \eqref{eq-D}]
	\ 
	
	Since $-\pi/2\leq q(\si)\leq \pi/2$ and 
	\[  \|u\|_{\Lf}+\|\nab\phi\|_{\Lf}\les \ep\,,   \]
	then we have
	\begin{align*}
		\EE_\al(t)+\int_0^t \int_{\R^2} e^{-q} |\nab\Ga^\al u|^2\ dxds+\tilde{\EE}_{\al}(\phi,t)\leq e^{\pi/2}E_\al(t)+(\ep e^{\pi/2} +2\ep^2 e^{\pi/2})E_\al(t)\,,
	\end{align*}
and 
\begin{align*}
	\EE_\al(t)+\int_0^t \int_{\R^2} e^{-q} |\nab\Ga^\al u|^2\ dxds+\tilde{\EE}_{\al}(\phi,t)\geq \frac{1}{2} e^{-\pi/2}E_\al(t)-(\ep e^{\pi/2} +2\ep^2 e^{\pi/2})E_\al(t)\,,
\end{align*}
Then the equivalence relation \eqref{eq} follows. Similarly, since $e^{-q(\si)}\approx 1$, we can also obtain the equivalence \eqref{eq-D}.
\end{proof}

The energy estimates will be provided in the following subsections.

\subsection{The energy estimates for velocity \texorpdfstring{$\Ga^\al u$}{}}

\begin{proof}[Proof of energy estimates \eqref{E-u}]
Taking the $L^2$ inner product of \eqref{Main_Sys_VecFie} with $e^{-\rho}\Ga^\al u$ and using integration by parts, we have
\begin{align*}
	\frac{d}{dt}\int_{\R^2} \frac{1}{2} e^{-q} |\Ga^\al u|^2\ dx
	&=\int_{\R^2} e^{-q}\Ga^\al u\ \d_t\Ga^\al u-\frac{e^{-q}}{2(1+\si^2)}|\Ga^\al u|^2\ dx\\
	&=\int_{\R^2} -\frac{1}{2}e^{-q}\Big(2|\nab\Ga^\al u|^2+\frac{2\om_j}{1+\si^2}\ \Ga^\al u\ \d_j\Ga^\al u+\frac{1}{1+\si^2}|\Ga^\al u|^2\Big)+e^{-q}\Ga^\al u \ \tilde f_\al\ dx\\
	&=\int_{\R^2} -\frac{1}{2}e^{-q}\Big((2|\nab\Ga^\al u|^2-\frac{|\d_r\Ga^\al u|^2}{1+\si^2})+\frac{|\Ga^\al u+\d_r\Ga^\al u|^2}{1+\si^2}\Big)+e^{-q}\Ga^\al u\ \tilde f_\al\ dx\,.
\end{align*}
where $f_\al$ is in \eqref{fal} and 
\[ \tilde f_\al= \De (  (S-1)^{\al_1}\Ga^{\al'}u-\Ga^\al u)+\P f_\al\,. \]
Since $2-(1+\si^2)^{-1}\geq 1$, this implies
\begin{align*}
	\frac{d}{dt}\int_{\R^2} \frac{1}{2} e^{-q} |\Ga^\al u|^2\ dx+\int_{\R^2} \frac{1}{2}e^{-q}\Big(|\nab\Ga^\al u|^2+\frac{|\Ga^\al u+\d_r\Ga^\al u|^2}{1+\si^2}\Big)\ dx\leq \int_{\R^2} e^{-q}\Ga^\al u\cdot \tilde f_\al\ dx\,.
\end{align*}
By the definition of $\DD_\al^2$ \eqref{Dal} and the relation
\begin{align*}
	\frac{1}{4}\Big(\frac{|\Ga^\al u|^2}{1+\si^2}+|\nab\Ga^\al u|^2\Big)\leq \frac{1}{2}\frac{|\Ga^\al u+\d_r \Ga^\al u|^2}{1+\si^2}+|\nab\Ga^\al u|^2\,,
\end{align*}
we further have
\begin{align*}
	\frac{d}{dt}\int_{\R^2} \frac{1}{2} e^{-q} |\Ga^\al u|^2\ dx+\frac{1}{4}\mathscr D_\al^2(u,t)\leq \int_{\R^2} e^{-q}\Ga^\al u\cdot \tilde f_\al\ dx\,.
\end{align*}

In view of the formula \eqref{fal} for $f_\al$, the right-hand side of the above can be written as 
\begin{align*}
	\int_{\R^2} e^{-q}\Ga^\al u\cdot \tilde f_\al\ dx
	&=\int_{\R^2}e^{-q} \Ga^\al u\cdot \De (  (S-1)^{\al_1}\Ga^{\al'}u-\Ga^\al u)\ dx\\
	&\quad  -\sum_{\be+\ga=\al}C_{\al}^\be\int_{\R^2}e^{-q} \Ga^\al u\cdot \P(\Ga^\be u\cdot\nab \Ga^\ga u)\ dx\\
	&\quad  -\sum_{\be+\ga=\al}C_\al^\be\int_{\R^2}e^{-q} \Ga^\al u\cdot\P\d_j(\nab\Gamma^{\be}\phi\cdot\d_j \Gamma^{\ga}\phi)\ dx\\
	&=I_1+I_2+I_3\,.
\end{align*}
Next, we estimate these terms respectively.
	
For the first integral $I_1$, by integration by parts and H\"{o}lder, we have
\begin{equation*}       \label{Ev_I1}
	\begin{aligned}
		I_1(s) &\lesssim  \sum_{|\be|<|\al|}\Big|\int_{\R^2} e^{-q}(  \nab \Ga^\al u \nab\Ga^\be u+\om_j(1+\si^2)^{-1}\Ga^\al u \d_j\Ga^\be u)dx\Big|\\
		&=\sum_{|\be|<|\al|}\Big|\int_{\R^2} e^{-q}\Big(|\nab\Ga^\al u||\nab\Ga^\be u|+ \frac{|\Ga^\al u|}{1+\si^2} |\nab\Ga^\be u| \Big)dx\Big|\\
		&\leq \sum_{|\be|<|\al|} \frac{1}{16} \int_{\R^2} e^{-q} \Big(\frac{|\Ga^\al u|^2}{1+\si^2}+|\nab\Ga^\al u|^2+C|\nab\Ga^\be u|^2\Big)\ dx\\
		&\leq  \frac{1}{8}\mathscr D^2_\al(u)+C\mathscr D^2_{|\al|-1}(u)\,.
	\end{aligned}
\end{equation*}
For the second term $I_2$, by integration by parts and Sobolev embedding we have
\begin{align*}
	I_2&\les \sum_{\be+\ga=\al} \|\nab(e^{-q}\Ga^\al u)\|_{L^2} \|\Ga^\be u\|_{L^4} \|\Ga^\ga u\|_{L^4} \\
	&\les \|\nab(e^{-q}\Ga^\al u )\|_{L^2}\|\Ga^\be u\|_{L^2}^{1/2}\|\nab\Ga^\be u\|_{L^2}^{1/2}\|\Ga^\ga u\|_{L^2}^{1/2}\|\nab \Ga^\ga u\|_{L^2}^{1/2}\\
	&\les \big(\big\|\frac{\Ga^\al u}{\<\si\>}\big\|_{L^2}+\|\nab \Ga^\al u\|_{L^2}\big)\big( \ep \mathscr D_\al^{1/2}(u) \mathscr D^{1/2}_{0}(u)+\ep \mathscr D_{|\al|-1}(u)\big)\,,
\end{align*}
which yields
\begin{align*}
	I_2(s) &\les \mathscr D_\al(u) \big(\ep \mathscr D_\al(u)+\ep \mathscr D_{|\al|-1}(u)\big)\les \ep \mathscr D_\al^2(u)+\ep \mathscr D_{|\al|-1}^2(u)\,.
\end{align*}
Finally, we bound the term $I_3$. By integration by parts,  $I_3$ can be written as 
\begin{align*}
	I_3&=-\sum_{\be+\ga=\al}C_\al^\be\int_{\R^2}e^{-q} \Ga^\al u\cdot \P\d_j(\nab \Gamma^{\be}\phi\cdot\d_j\Gamma^{\ga}\phi)dx\\
	&=\sum_{\be+\ga=\al}C_\al^\be\int_{\R^2}e^{-q} \Big( \frac{\om_j}{\<\si\>^2}\Ga^\al u+ \d_j\Ga^\al u\Big)\cdot \P(\nab \Gamma^{\be}\phi\cdot\d_j\Gamma^{\ga}\phi)\  dx
\end{align*} 
By \eqref{AwayCone0} and \eqref{NearCone0}, the integrand can be bounded by 
\begin{align*}
	&\|\P(\nab \Gamma^{\be}\phi\cdot\d_j\Gamma^{\ga}\phi)\|_{L^2}
	\les E^{1/2}_N \|\nab \Ga^{[N/2]}\phi\|_{\Lf}\les \<t\>^{-1/2}E^{1/2}_N (E^{1/2}_{[N/2]+1}+\XX^{1/2}_{[N/2]+2})\,.
\end{align*}
Then we have
\begin{align*}
	I_3&\les  \Big(\Big\| \frac{\Ga^\al u}{\<\si\>^2}\Big\|_{L^2}+ \|\nab\Ga^\al u\|_{L^2}\Big)\|\P(\nab \Gamma^{\be}\phi\cdot\d_j\Gamma^{\ga}\phi)\|_{L^2}\\
	&\les \ep^2\<t\>^{-1/2}\mathscr D_\al(u) \les \ep \mathscr D_\al^2(u)+\ep^3 \<t\>^{-1}\,.
\end{align*}

As a consequence of the above estimates of $I_1,I_2,I_3$ and (\ref{MainAss_dini}), we have
\begin{equation*}
	\frac{d}{dt}\EE_\al(u,t)+\frac{1}{4}\mathscr D_\al^2(u,t)\leq \big(\frac{1}{8}+C\ep\big)\mathscr D_\al^2(u,t)+C\mathscr D_{|\al|-1}^2(u,t)+ \ep^3 \<t\>^{-1}\,. 
\end{equation*}
The first term in the right hand side can be absorbed by the left hand side, then the estimate \eqref{E-u} is obtained. This completes the proof of \eqref{E-u}.
\end{proof}

\medskip 
\subsection{The energy estimates for orientation fields \texorpdfstring{$\Ga^\al\phi$}{}} 

\begin{proof}[Proof of energy estimates \eqref{E-phi}]
	By the $\Ga^\al\phi$-equation in (\ref{Main_Sys_VecFie}) we calculate
		\begin{align*}
		\frac{d}{dt}\int_{\R^2} \frac{1}{2} e^{-q} |\nab_{t,x}\Ga^\al \phi|^2\ dx
		&=\int_{\R^2} e^{-q} \Big(\d_t \Ga^\al\phi \d^2_t \Ga^\al\phi +\nab\Ga^\al\phi \nab\d_t\Ga^\al\phi-\frac{1}{2(1+\si^2)}|\nab_{t,x}\Ga^\al\phi|^2\Big) \ dx \\
		&=\int_{\R^2} -\frac{e^{-q}}{2(1+\si^2)} \big(2\om_j\d_t\Ga^\al\phi\ \d_j\Ga^\al\phi+|\d_t\Ga^\al\phi|^2+|\nab\Ga^\al\phi|^2\big)+e^{-q} \d_t\Ga^\al\phi\cdot g_\al\ dx \\
		&=-\int_{\R^2} \frac{e^{-q}}{2(1+\si^2)} |\om\d_t\Ga^\al\phi+\nab\Ga^\al\phi|^2\ dx+ \int_{\R^2} e^{-q} \d_t\Ga^\al\phi\cdot g_\al\ dx\,.
	\end{align*}
In view of the formula \eqref{gal} for $g_\al$, this implies
\begin{align*}
	\frac{d}{dt}\EE_\al(\phi,t)+\mathscr D_\al^2(\phi,t)&= \int_{\R^2} e^{-q}\d_t\Ga^\al\phi\cdot g_\al\ dx\\
	&= -2\sum_{\be+\ga=\al}C_\al^\be \int_{\R^2} e^{-q}\d_t\Ga^\al\phi\cdot (\Ga^\be u\cdot\nab\d_t\Ga^\ga\phi)\ dx\\
	&\quad -\sum_{\be+\ga=\al}C_\al^\be \int_{\R^2}e^{-q}\d_t\Ga^\al\phi\cdot (\d_t \Ga^\be u\cdot\nab\Ga^\ga\phi)\ dx\\
	&\quad-\sum_{\be+\ga+\mu=\al}C_\al^{\be,\ga} \int_{\R^2}e^{-q}\d_t\Ga^\al\phi\cdot \big(\Ga^\be u\cdot\nab(\Ga^\ga u\cdot\nab\Ga^\mu\phi)\big)\ dx\\
	:&= J_1+J_2+J_3\,.
\end{align*}
It remains to bound the nonlinear terms $J_i$ for $i=1,\ 2,\ 3$.

\medskip 
\emph{Step 1. Estimate of the term $J_1$} 
\begin{align} \label{J1}
	J_1\les \ep \mathscr D_N^2(t)+\ep^3\<t\>^{-1}\,.
\end{align}

\medskip
\emph{Case i). $|\be|\geq [N/2]$.} 

Here we would utilize the ghost weight energy of ``$\Ga^\al u$" and the decay of $D^2 \Ga^\ga\phi$ to control $J_1$. Precisely, when $r$ stays away from the light cone, i.e. $r\leq \frac{2}{3}\<t\>$ or $r\geq \frac{5}{4}\<t\>$, by \eqref{AwayCone} and \eqref{XX_ka-2}, this term $J_1$ is bounded by 
\begin{align*}
	&\sum_{\be+\ga=\al;|\be|\geq [N/2]} \int_{r\notin[ \frac{2}{3}\<t\>,\frac{5}{4}\<t\>]} e^{-q}\d_t\Ga^\al\phi\cdot (\Ga^\be u\cdot\nab\d_t\Ga^\ga\phi)\ dx\\
	&\lesssim \|\d_t \Ga^N\phi\|_{L^2}\|\Ga^N u\|_{L^2} \<t\>^{-1}\|\<t\>\nab\d_t\Ga^{[N/2]}\phi\|_{L^\infty(r\notin[ \frac{2}{3}\<t\>,\frac{5}{4}\<t\>])}\\
	&\lesssim \<t\>^{-1}\EE_{N}\mathcal X^{1/2}_{[N/2]+3}\les \<t\>^{-1}\ep^3\,.
\end{align*}
When $r$ is near the light cone, i.e. $\frac{\<t\>}{3}\leq r\leq \frac{5}{2}\<t\>$, by \eqref{NearCone} and \eqref{XX_ka-2}, we have
\begin{align*}
	&\sum_{\be+\ga=\al;|\be|\geq [N/2]} \int_{\frac{\<t\>}{3}\leq r\leq \frac{5}{2}\<t\>} e^{-q}\d_t\Ga^\al\phi\cdot (\Ga^\be u\cdot\nab\d_t\Ga^\ga\phi)\ dx\\
	&\lesssim \int_{\frac{\<t\>}{3}\leq r\leq \frac{5}{2}\<t\>} \<t\>^{-1/2} e^{-q} \d_t\Ga^N\phi \ \frac{|\Ga^\be u|}{(1+\si^2)^{1/2}}\cdot r^{1/2}(1+\si^2)^{1/2}\nab\d_t\Ga^{[N/2]}\phi\ dx\\
	&\lesssim \<t\>^{-1/2} \|\d_t\Ga^N \phi\|_{L^2} \mathscr D_{|\al|}\|r^{1/2}\<t-r\>\nab\d_t\Ga^{[N/2]}\phi\|_{L^\infty(\frac{\<t\>}{3}\leq r\leq \frac{5}{2}\<t\>)}\\
	&\lesssim \<t\>^{-1/2} \EE^{1/2}_{N} \mathscr D_{|\al|}(\mathcal E^{1/2}_{[N/2]+2}+\mathcal X^{1/2}_{[N/2]+3})\\
	&\les \ep^2 \<t\>^{-1/2}\mathscr D_{|\al|}\les \ep \mathscr D_{|\al|}^2+\ep^3\<t\>^{-1}\,.
\end{align*}

\medskip 

\emph{Case ii). $|\be|<[N/2]$.} 

Here we would use $\div\  u=0$ and the structure of $\Ga^\be u\cdot \nab\d_t\Ga^\ga\phi$. Precisely, when $\be=0$, by integration by parts, divergence free of $ u$ and \eqref{rv}, we have
\begin{align*}
	&-\int_{\R^2} e^{-q}\d_t\Ga^\al\phi\cdot ( u\cdot\nab\d_t\Ga^\al\phi)\ dx=\int \frac{1}{2} e^{-q} \frac{\om_j u_j}{1+\si^2} |\d_t\Ga^\al\phi|^2\ dx \\
	& \lesssim \int_{r<\frac{2}{3}\<t\>} \<t\>^{-2} |u||\d_t\Ga^\al\phi|^2\ dx +\int_{r\geq \frac{2}{3}\<t\>} \frac{r}{\<t\>}|\om_j u_j||\d_t\Ga^\al\phi|^2\ dx\\
	&\lesssim (\<t\>^{-2}\|u\|_{L^\infty}+\<t\>^{-1}\|r\om\cdot u\|_{L^\infty})\|\d_t\Ga^\al\phi\|_{L^2}^2\\
	&\lesssim \<t\>^{-1}\EE^{1/2}_{N-2}\EE_N\les \ep^3 \<t\>^{-1}\,.
\end{align*}

When $1\leq |\be|<[N/2]$, we bound $J_1$ in the regions $r<\frac{2}{3}\<t\>$ and $r\geq \frac{2}{3}\<t\>$, respectively. Precisely, by \eqref{Decay-u} and \eqref{XX_ka}, we have
\begin{align*}
	&\sum_{\be+\ga=\al;1\leq |\be|<[N/2]}\int_{r<\frac{2}{3}\<t\>} e^{-q} \d_t\Ga^\al\phi \Ga^\be u\cdot\nab\d_t\Ga^\ga \phi\ dx\\
	&\lesssim \int_{r<\frac{2}{3}\<t\>} \<t\>^{-1}  |\d_t\Ga^N\phi ||\Ga^{[N/2]} u|\cdot | \<t-r\>\nab\d_t\Ga^{N-1} \phi|\ dx\\
	&\lesssim \<t\>^{-1} \EE^{1/2}_N\|\Ga^{[N/2]} u\|_{L^\infty} \XX^{1/2}_{N}\\
	&\lesssim \<t\>^{-3/2+\de}\EE^{1/2}_N\ep \mathcal X^{1/2}_{N}\\
	&\les \<t\>^{-3/2+\de}\ep^2(\ep+\ep \<t\>^{1/2} \mathscr D_{N})\\
	&\les \ep^3 \<t\>^{-3/2+\de}+\ep \mathscr D_N^2\,.
\end{align*}
For the region $r\geq \frac{2}{3}\<t\>$, we use the decomposition of $\nab$ to split the integral into
\begin{align*}
	\int_{r\geq \frac{2}{3}\<t\>} \d_t\Ga^\al\phi \Ga^\be u\cdot\nab\d_t\Ga^\ga \phi\ dx &\lesssim \int_{r\geq \frac{2}{3}\<t\>}   \d_t\Ga^\al\phi\  \Ga^\be u\cdot \big(\om\d_r+\frac{x^{\bot}}{r^2}\Om\big)\d_t\Ga^\ga \phi\ dx\,.
\end{align*}
By \eqref{omf} we have
\begin{align*}
	\sum_{\be+\ga=\al;1\leq |\be|<[N/2]}\int_{r\geq \frac{2}{3}\<t\>} \d_t\Ga^\al\phi\ \om\cdot \Ga^\be u\d_r\d_t\Ga^\ga \phi\ dx &\lesssim \int_{r\geq \frac{2}{3}\<t\>} \<t\>^{-1}  |\d_t\Ga^N\phi|\  |r \om\cdot\Ga^{[N/2]} u| |\d_r\d_t\Ga^{N-1} \phi|\ dx\\
	&\lesssim \<t\>^{-1} \EE_N\|r\om\cdot\Ga^{[N/2]} u\|_{L^\infty} \\
	&\lesssim \<t\>^{-1} \EE_{N}\EE^{1/2}_{[N/2]+2}\les \ep^3\<t\>^{-1} \,.
\end{align*}
And by \eqref{Decay-u} we easily have
\begin{align*}
	\sum_{\be+\ga=\al;1\leq |\be|<[N/2]}\int_{r\geq \frac{2}{3}\<t\>} \d_t\Ga^\al\phi\  \Ga^\be u\cdot\frac{x^\bot}{r^2}\Om\d_t\Ga^\ga \phi\ dx &\lesssim \int_{r\geq \frac{2}{3}\<t\>} \<t\>^{-1}  |\d_t\Ga^N\phi|\  |\Ga^{[N/2]} u|\  |\d_t\Om\Ga^{N-1} \phi|\ dx\\
	&\lesssim \<t\>^{-1} \EE_N\|\Ga^{[N/2]}u\|_{L^\infty}\\
	&\lesssim \<t\>^{-3/2+\de} \EE_{N}\ep\\
	&\les \ep^3\<t\>^{-3/2+\de} \,.
\end{align*}

Hence, by the above bounds we obtain the estimate \eqref{J1}

\bigskip

\emph{Step 2. Estimate of $J_2$}
\begin{align*}% \label{J2}
	J_2+\frac{d}{dt}\int_{\R^2} e^{-q}\d_t\Ga^\al\phi\Ga^\al u\cdot\nab\phi\ dx+\frac{1}{2}\frac{d}{dt}\int_{\R^2} e^{-q} |\Ga^\al u\cdot \nab\phi|^2 \ dx\les \ep \mathscr D_N^2(t) +\ep^3\<t\>^{-1}\,.
\end{align*}

Here we would bound $J_2$ case by case. Due to the derivative loss, we should rewrite the $J_2$ with $\be=\al$. Precisely, by integration by parts and the $\Ga^\al \phi$-equations in \eqref{Main_Sys_VecFie}, we have
\begin{align*}
	&-\int_{\R^2}e^{-q}\d_t\Ga^\al\phi\cdot (\d_t \Ga^\al u\cdot\nab\phi)\ dx\\
	&= -\frac{d}{dt}\int_{\R^2} e^{-q}\d_t\Ga^\al\phi\Ga^\al u\cdot\nab\phi\ dx-\int_{\R^2} e^{-q}(1+\si^2)^{-1}\d_t\Ga^\al\phi\Ga^\al u\cdot\nab\phi\ dx\\
	&\quad +\int_{\R^2} e^{-q}\d^2_t\Ga^\al\phi\Ga^\al u\cdot\nab\phi\ dx+\int_{\R^2} e^{-q}\d_t\Ga^\al\phi\Ga^\al u\cdot\nab\d_t\phi\ dx\\
	&=-\frac{d}{dt}\int_{\R^2} e^{-q}\d_t\Ga^\al\phi\Ga^\al u\cdot\nab\phi\ dx-\int_{\R^2} e^{-q}(1+\si^2)^{-1}\d_t\Ga^\al\phi\Ga^\al u\cdot\nab\phi\ dx\\
	&\quad +\int_{\R^2} e^{-q}\d_t\Ga^\al\phi\Ga^\al u\cdot\nab\d_t\phi\ dx\\
	&\quad +\int_{\R^2} e^{-q} \De \Ga^\al \phi \Ga^\al u\cdot \nab\phi\ dx -\sum_{\be+\ga=\al}C_\al^\be\int_{\R^2} e^{-q}(\d_t \Ga^\be u\cdot\nab\Ga^\ga\phi)\Ga^\al u\cdot \nab\phi\ dx\\
	&\quad -2\sum_{\be+\ga=\al}C_\al^\be\int_{\R^2} e^{-q} (\Ga^\be u\cdot\nab\d_t\Ga^\ga\phi )\Ga^\al u\cdot\nab\phi\ dx\\
	&\quad-\sum_{\be+\ga+\mu=\al}C_\al^{\be,\ga}\int_{\R^2} e^{-q} \ \big(\Ga^\be u\cdot\nab(\Ga^\ga u\cdot\nab\Ga^\mu\phi)\big)\Ga^\al u\cdot \nab\phi\ dx\\
	&=J_{21}+\cdots+ J_{27}\,.
\end{align*}
We also record the $J_2$ with $|\be|\leq |\al|-1$ as 
\begin{align*}
	J_{28}=-\sum_{\be+\ga=\al;|\be|\leq |\al|-1}C_\al^\be \int_{\R^2}e^{-q}\d_t\Ga^\al\phi\cdot (\d_t \Ga^\be u\cdot\nab\Ga^\ga\phi)\ dx\,.
\end{align*}
The first term $J_{21}$ is retained, which can be moved to the left hand side. Next, we bound the other terms respectively.

\medskip 

\emph{Case i): Estimate of $J_{28}$.} 

By the decays \eqref{decay-dtu}, \eqref{dtu-L2al}, \eqref{AwayCone0} and \eqref{NearCone0}, we have 
\begin{align*}
	J_{28}
	\lesssim &\  \|\d_t\Ga^N\phi\|_{L^2}\|\d_t\Ga^1 u\|_{\Lf}\|\nab\Ga^{N}\phi\|_{L^2}+\|\d_t\Ga^N \phi\|_{L^2}\|\d_t\Ga^{[N/2]}u\|_{L^2}\|\nab\Ga^{N-2}\phi\|_{\Lf}\\
	&+\|\d_t\Ga^N \phi\|_{L^2}\|\d_t\Ga^{N-1}u\|_{L^2}\|\nab\Ga^{[N/2]}\phi\|_{\Lf}\\
	\les &\ \ep^3\<t\>^{-1}+\ep (\mathscr D_{[N/2]+1}+\ep\<t\>^{-1/2}) \<t\>^{-1/2}(\EE^{1/2}_{N-1}+\XX^{1/2}_{N})\\
	&  +\ep(\mathscr D_{N}+\ep\<t\>^{-1/2}) \<t\>^{-1/2}(\EE^{1/2}_{[N/2]+1}+\XX^{1/2}_{[N/2]+2})\,.
\end{align*} 
Then by \eqref{XX_ka}, we further bound this by 
\begin{align*}
	J_{28}&\les \ep^3\<t\>^{-1}+\ep(\mathscr D_N+\ep\<t\>^{-1/2})(\ep\<t\>^{-1/2}+\ep \mathscr D_N)+\ep^2\<t\>^{-1/2}(\mathscr D_N+\ep\<t\>^{-1/2})\\
	&\les \ep \mathscr D^2_N +\ep^3\<t\>^{-1}\,.
\end{align*}

\medskip 

\emph{Case ii): Estimates of $J_{22}$ and $J_{23}$.}

For the second term $J_{22}$, by the ghost weight energy of $``\Ga^\al u"$, \eqref{AwayCone0} and \eqref{NearCone0}, we have
\begin{align*}
	J_{22}&\lesssim \int_{\R^2} e^{-q}|\d_t\Ga^\al \phi|\ \Big|\frac{\Ga^\al u}{(1+\si^2)^{1/2}}\Big||\nab\phi|\ dx\\
	&\lesssim \|\d_t\Ga^N \phi\|_{L^2} \mathscr D_{\al}(u)\|\nab\phi\|_{L^\infty}\\
	&\lesssim \ep^2 \<t\>^{-1/2}\mathscr D_{\al}(u)\les \ep \mathscr D_\al^2(u) +\ep^3\<t\>^{-1}\,.
\end{align*}
The last term $II_{23}$ has been estimated in the \emph{Step 1. Case i)} above. Namely, it can also be bounded by
\begin{equation*}
	J_{23}\les \ep \mathscr D_N^2 +\ep^3\<t\>^{-1}\,.
\end{equation*}

\medskip 

\emph{Case iii). Estimates of $J_{24}$.}

Integration by parts yields
\begin{align*}
	J_{24}&=-\int_{\R^2} e^{-q} \nab\Ga^\al\phi \nab\Ga^\al u\cdot\nab\phi\ dx-\int_{\R^2} e^{-q}\nab\Ga^\al\phi\ \Ga^\al u\cdot\nab\nab\phi\ dx\\
	&\quad -\int_{\R^2} e^{-q} \d_j\Ga^\al\phi \frac{\om_j \Ga^\al u}{1+\si^2}\cdot \nab\phi\ dx \,. 
\end{align*}
The second term can be estimated similar to \emph{Step 1. Case i)}. For the other two terms, by \eqref{AwayCone0} and \eqref{NearCone0} we have
\begin{align*}
	&-\int_{\R^2} e^{-q} \nab\Ga^\al\phi \nab\Ga^\al u\cdot\nab\phi\ dx -\int_{\R^2} e^{-q} \d_j\Ga^\al\phi \frac{\om_j \Ga^\al u}{1+\si^2}\cdot \nab\phi\ dx\\
	&\lesssim \|\nab\Ga^\al\phi\|_{L^2}\mathscr D_{\al}(u) \|\nab\phi\|_{L^\infty} 
	\lesssim \ep^2 \<t\>^{-1/2} \mathscr D_{\al}(u)\les \ep \mathscr D_\al^2(u) +\ep^3\<t\>^{-1}\,. 
\end{align*}

\medskip 

\emph{Case iv). Estimates of $J_{25}$.} 

For the case $\be=\al$, by integration by parts we have
\begin{align*}
	&-\int_{\R^2} e^{-q}(\d_t\Ga^\al u\cdot \nab\phi) \Ga^\al u\cdot\nab\phi\ dx\\
	&=-\int_{\R^2} e^{-q} \d_t(\Ga^\al u\cdot \nab\phi) \Ga^\al u\cdot\nab\phi\ dx+\int_{\R^2} e^{-q} (\Ga^\al u\cdot \nab\d_t\phi) \Ga^\al u\cdot\nab\phi\ dx\\
	&=-\frac{1}{2}\frac{d}{dt}\int_{\R^2} e^{-q} |\Ga^\al u\cdot \nab\phi|^2 \ dx-\int_{\R^2} \frac{1}{2(1+\si^2)} e^{-q} |\Ga^\al u\cdot \nab\phi|^2 \ dx+\int_{\R^2} e^{-q} (\Ga^\al u\cdot \nab\d_t\phi) \Ga^\al u\cdot\nab\phi\ dx\,.
\end{align*}
The first term is retained, which will be moved to the left hand side.
Then we use \eqref{AwayCone0} and \eqref{NearCone0} to bound the other two terms by
\begin{align*}
	&-\int_{\R^2} \frac{1}{2(1+\si^2)} e^{-q} |\Ga^\al u\cdot \nab\phi|^2 \ dx+\int_{\R^2} e^{-q} (\Ga^\al u\cdot \nab\d_t\phi) \Ga^\al u\cdot\nab\phi\ dx\\
	&\lesssim \mathscr D_\al^2(u) \|\nab\phi\|_{\Lf}^2
	+ \|\Ga^\al u\|_{L^2}^2 \|\mathbf{1}_{r\notin[\frac{2}{3}\<s\>,\frac{5}{4}\<s\>]}(r) \nab\d_t\phi \nab\phi\|_{\Lf}\\
	&\quad +\mathscr D_{\al}(u)\EE_N^{1/2}\|\mathbf{1}_{r\in[\frac{2}{3}\<t\>,\frac{5}{4}\<t\>]}(r) \<t-r\>\nab\d_t\phi\nab\phi\|_{\Lf}\\
	&\lesssim \ep^2 \<t\>^{-1}\mathscr D_\al^2(u)+\ep^4\<t\>^{-2+2\de}+\mathscr D_\al(u) \ep^3 \<t\>^{-1}\\
	&\lesssim \ep^2 \mathscr D_\al^2(u)+\ep^4\<t\>^{-2+2\de}\,.
\end{align*}
For the case $|\be|\leq  |\al|-1$, by \eqref{decay-dtu}, \eqref{dtu-L2al}, Lemma \ref{Decay-phi0} and \eqref{XX_ka}, we have
\begin{align*}
	&-\sum_{\be+\ga=\al;|\be|\leq |\al|-1}C_\al^\be\int_{\R^2} e^{-q}(\d_t \Ga^\be u\cdot\nab\Ga^\ga\phi)\Ga^\al u\cdot \nab\phi\ dx\\
	&\lesssim  \|\d_t\Ga^1 u\|_{\Lf}\|\nab\Ga^N\phi\|_{L^2}\|\Ga^N u\|_{L^2}\|\nab\phi\|_{\Lf} \\
	&\quad + \|\d_t\Ga^{N-1} u\|_{L^2}\|\nab\Ga^{N-2}\phi\|_{\Lf}\|\Ga^N u\|_{L^2}\|\nab\phi\|_{\Lf} \\
	&\lesssim \ep^4 \<t\>^{-3/2}+(\mathscr D_{N}+ \<t\>^{-1/2}\EE^{1/2}_{N})\<t\>^{-1}(\EE^{1/2}_N+\XX^{1/2}_N)\EE_N\\
	&\lesssim \ep^4 \<t\>^{-3/2}+(\mathscr D_{N}+ \ep\<t\>^{-1/2})\ep^2\<t\>^{-1}(\ep+\ep\<t\>^{1/2}\mathscr D_N)\\
	&\lesssim \ep^4 \<t\>^{-3/2}+\ep^3\<t\>^{-1}\mathscr D_N+\ep^3\<t\>^{-1/2}\mathscr D_N^2\\
	&\les \ep \mathscr D_N^2+\ep^4 \<t\>^{-3/2}\,.
\end{align*}
Thus we also obtain
\begin{equation*}
	J_{25}+\frac{1}{2}\frac{d}{dt}\int_{\R^2} e^{-q} |\Ga^\al u\cdot \nab\phi|^2 \ dx\lesssim  \ep \mathscr D_N^2+\ep^4 \<t\>^{-3/2}\,.
\end{equation*}

\medskip 

\emph{Case v). Estimates of $J_{26}$.}

Due to the derivative loss, we rewrite the term $J_{26}$ as 
\begin{align*}
	J_{26}&=-2\sum_{\be+\ga=\al;|\ga|<|\al|}C_\al^\be\int_{\R^2} e^{-q} (\Ga^\be u\cdot\nab\d_t\Ga^\ga\phi )\Ga^\al u\cdot\nab\phi\ dx-2\int_{\R^2} e^{-q} ( u\cdot\nab\d_t\Ga^\al\phi )\Ga^\al u\cdot\nab\phi dx\\
	&= J_{26}^1+J_{26}^2\,.
\end{align*}

For the first term $J_{26}^1$, by \eqref{AwayCone0} and \eqref{NearCone0}, we have
\begin{align*}
	J_{26}^1&\lesssim \sum_{\be+\ga=\al;|\ga|\leq [N/2]}\|\Ga^\be u\|_{L^2}\|\Ga^\al u\|_{L^2} \|\mathbf{1}_{r\notin [\frac{2}{3}\<t\>,\frac{5}{4}\<t\>]}\nab\d_t\Ga^\ga\phi\ \nab\phi\|_{\Lf}\\
	&\quad +\sum_{\be+\ga=\al;|\ga|\leq [N/2]}\Big\|\frac{\Ga^\be u}{(1+\si^2)^{1/2}}\Big\|_{L^2}\|\Ga^\al u\|_{L^2} \|\mathbf{1}_{r\in [\frac{2}{3}\<t\>,\frac{5}{4}\<t\>]}(1+\si^2)^{1/2}\nab\d_t\Ga^\ga\phi\ \nab\phi\|_{\Lf}\\
	&\quad+\sum_{\be+\ga=\al; [N/2]<|\ga|<|\al|}\|\Ga^\be u\|_{\Lf}  \|\nab\d_t\Ga^\ga\phi\|_{L^2} \|\Ga^\al u\|_{L^2}\|\nab\phi\|_{\Lf(r\notin[2\<t\>/3,5\<t\>/4])}\\
	&\quad+\sum_{\be+\ga=\al;[N/2]<|\ga|<|\al|}\|\Ga^\be u\|_{\Lf}\Big\|\frac{\Ga^\al u}{(1+\si^2)^{1/2}}\Big\|_{L^2}\|\mathbf{1}_{(r\in[2\<t\>/3,5\<t\>/4])}(1+\si^2)^{1/2}\nab\d_t\Ga^\ga\phi\|_{L^2}\| \nab\phi\|_{\Lf}\,.
\end{align*}
Then by \eqref{AwayCone0}, \eqref{NearCone0}, \eqref{Decay-u}, \eqref{XX} and \eqref{XX_ka}, we bound the above terms by 
\begin{align*}
	J_{26}^1&\lesssim \EE_N^2\<t\>^{-2+\de}+\mathscr D_{N}\EE_N^{1/2} \<t\>^{-1} \EE_N+\ep\<t\>^{-1/2+\de}\ep^2 \ep\<t\>^{-1+\de}+\ep\<t\>^{-1/2+\de}\mathscr D_N\XX^{1/2}_{N}\ep\<t\>^{-1/2}\\
	&\lesssim \ep^4\<t\>^{-2+\de}+\ep^3\<t\>^{-1}\mathscr D_N+\ep^4\<t\>^{-3/2+2\de}+\ep^2\<t\>^{-1+\de}\mathscr D_N(\ep+\ep\<t\>^{1/2}\mathscr D_N)\\
	&\les \ep^4\<t\>^{-3/2+2\de}+\ep^2 \mathscr D_N^2\,.
\end{align*}

For the term $J_{26}^2$, by integration by parts in spatial $x$, we rewrite it as 
\begin{align*}
	J_{26}^2=&\ 2\int_{\R^2} e^{-q} \frac{\om_j}{1+\si^2} ( u_j\cdot\d_t\Ga^\al\phi )\Ga^\al u\cdot\nab\phi dx+2\int_{\R^2} e^{-q}  ( u\cdot\d_t\Ga^\al\phi )\nab \Ga^\al u\cdot\nab\phi dx\\
	& +2\int_{\R^2} e^{-q}  ( u\ \d_t\Ga^\al\phi )\Ga^\al u\cdot\nab \nab\phi dx\,.
\end{align*}
Then we use \eqref{omf}, \eqref{Decay-u}, \eqref{AwayCone0}, \eqref{NearCone0} to bound this by
\begin{align*}
	J_{26}^2\lesssim&\ \Big\|\frac{\om\cdot u}{1+\si^2}\Big\|_{\Lf}\|\d_t\Ga^\al\phi\|_{L^2}\|\Ga^\al u\|_{L^2}\|\nab\phi\|_{\Lf}+\|u\|_{\Lf}\|\d_t\Ga^\al\phi\|_{L^2}\|\nab\Ga^\al u\|_{L^2}\|\nab\phi\|_{\Lf}\\
	& +\|u\|_{\Lf}\|\d_t\Ga^\al\phi\|_{L^2}\Big\|\frac{\Ga^\al u}{\<\si\>}\Big\|_{L^2}\<t\>^{-1/2}\|\mathbf{1}_{r\in[\frac{2}{3}\<t\>,\frac{5}{4}\<t\>]}r^{1/2}\<\si\>\nab^2\phi\|_{\Lf}\\
	& +\|u\|_{\Lf}\|\d_t\Ga^\al\phi\|_{L^2}\|\Ga^\al u\|_{L^2}\<t\>^{-1}\|\mathbf{1}_{r\notin[\frac{2}{3}\<t\>,\frac{5}{4}\<t\>]}t\nab^2\phi\|_{\Lf}\\
	\lesssim &\ \ep \<t\>^{-1}\ep^2 \ep\<t\>^{-1/2}+\ep\<t\>^{-1/2}\ep\mathscr D_{\al}\ep\<t\>^{-1/2}+\ep\<t\>^{-1/2}\ep\mathscr D_{\al} \ep\<t\>^{-1/2}+\ep\<t\>^{-1/2}\ep^2\ep\<t\>^{-1}\\
	\les &\ \ep^3\<t\>^{-1}\mathscr D_N+\ep^4 \<t\>^{-3/2}\\
	\les&\  \ep^2 \mathscr D_N^2+\ep^4 \<t\>^{-3/2}\,.
\end{align*}
Hence, the $J_{26}$ is bounded by 
\begin{equation*}
	J_{26}\les \ep^2 \mathscr D_N^2+\ep^4 \<t\>^{-3/2+2\de}\,.
\end{equation*}

\medskip 

\emph{Case vi). Estimates of $J_{27}$.}

When $\mu=\al$, by integration by parts, \eqref{Decay-u} and \eqref{AwayCone0} and \eqref{NearCone0}, we have
\begin{align*}
	&\int_{\R^2} e^{-q}u\cdot\nab(u\cdot \nab\Ga^\al\phi)\Ga^\al u\cdot\nab\phi\ dx\\
	&=\int_{\R^2} -e^{-q}\frac{\om_j}{1+\si^2} u_j(u\cdot\nab\Ga^\al\phi)\Ga^\al u\cdot\nab\phi+e^{-q} u_j(u\cdot\nab\Ga^\al\phi)\d_j(\Ga^\al u\cdot\nab\phi)\ dx\\
	&\lesssim \|u\|_{\Lf}^2\|\nab\Ga^\al\phi\|_{L^2}\|\<\nab\>\Ga^\al u\|_{L^2}\|\<\nab\>\nab\phi\|_{\Lf}\\
	&\lesssim \ep^4\<t\>^{-3/2}(\ep+\|\nab\Ga^\al u\|_{L^2})\\
	&\les \ep^5\<t\>^{-3/2}+\ep^2\mathscr D_N^2\,.
\end{align*}
When $\ga=\al$, we have
\begin{align*}
	&\int_{\R^2} e^{-q} \ (u\cdot\nab(\Ga^\al u\cdot\nab\phi))\Ga^\al u\cdot \nab\phi\ dx\\
	&=-\int_{\R^2} e^{-q} \ \frac{1}{2\<\si\>^2}\om\cdot u|\Ga^\al u\cdot\nab\phi|^2\ dx\\
	&\les \|u\|_{\Lf}\|\nab\phi\|_{\Lf}^2\|\Ga^\al u\|_{L^2}^2\\
	&\les \ep^5\<t\>^{-3/2}\,.
\end{align*}
For the remainder cases $|\ga|,|\mu|\leq |\al|-1$, we have
\begin{align*}
	&\sum_{\be+\ga+\de=\al;|\de|\leq |\al|-1}\int_{\R^2} e^{-q} \ (\Ga^\be u\cdot\nab(\Ga^\ga u\cdot\nab\Ga^\de\phi))\Ga^\al u\cdot \nab\phi\ dx\\
	&\lesssim  \|\Ga^{[N/2]} u\|_{\Lf}\|\Ga^{[N/2]+1} u\|_{\Lf} \|\nab\Ga^{N}\phi\|_{L^2}\|\Ga^N u\|_{L^2}\|\nab\phi\|_{\Lf}\\
	&\quad +\|\Ga^{[N/2]} u\|_{\Lf}\|\Ga^N u\|_{L^2} \|\nab\Ga^{[N/2]+1}\phi\|_{L^\infty}\|\Ga^N u\|_{L^2}\|\nab\phi\|_{\Lf}\\
	&\lesssim \ep^5 \<t\>^{-3/2+2\de}\,.
\end{align*}
Thus we obtain
\begin{align*}
	J_{27} \lesssim \ep^5\<t\>^{-3/2+2\de}+\ep^2\mathscr D_N^2\,.
\end{align*}

\bigskip
\emph{Step 3. We aim to show that}
\begin{align*}
	J_3+\frac{1}{2}\frac{d}{dt}\int_{\R^2} e^{-q}|u\cdot\nab\Ga^\al\phi|^2\ dx\les \ep^2\mathscr D_N^2+\ep^4\<t\>^{-3/2+\de}\,.
\end{align*}

\medskip 
\emph{Case i). Estimates of $J_3$ with $\mu=\al$.}

When $\mu=\al$, by integration by parts we have
\begin{align*}
	&\int_{\R^2}e^{-q}\d_t\Ga^\al\phi\cdot ( u\cdot\nab( u\cdot\nab\Ga^\al\phi))\ dx\\
	=&\ \int_{\R^2} e^{-q} \frac{\om_j}{1+\si^2} u_j \d_t\Ga^\al\phi (u\cdot\nab\Ga^\al\phi)
	+e^{-q}\d_t u\cdot\nab\Ga^\al\phi (u\cdot\nab\Ga^\al\phi)-e^{-q}\frac{|u\cdot\nab\Ga^\al\phi|^2}{1+\si^2} \ dx\\
	&\ -\frac{1}{2}\frac{d}{dt}\int_{\R^2} e^{-q}|u\cdot\nab\Ga^\al\phi|^2\ dx\\
	=&\ J_{31}+J_{32}\,.
\end{align*}
The term $J_{32}$ can be moved to the left hand side, so it suffices to bound $J_{31}$.

Due to Sobolev embedding,
\begin{align*}
	\Big\|\frac{u}{(1+\si^2)^{1/2}}\Big\|_{\Lf}\les \Big\|\frac{u}{(1+\si^2)^{1/2}}\Big\|_{H^2}\lesssim \Big\|\frac{ u}{(1+\si^2)^{1/2}}\Big\|_{L^2}+\|\nab u\|_{H^1}\,.
\end{align*}
Then the first term $J_{31}$ is controlled by 
\begin{align*}
	J_{31}&\lesssim \Big(\Big\|\frac{u}{(1+\si^2)^{1/2}}\Big\|_{L^2}+\|\nab u\|_{H^1}\Big)^2 \|D\Ga^\al\phi\|_{L^2}^2+\|\d_t u\|_{\Lf}\|u\|_{\Lf}\|\nab\Ga^\al\phi\|_{L^2}^2\\
	&\lesssim \ep^2\mathscr D_N^2+\ep^4\<t\>^{-3/2}\,.
\end{align*}

\medskip 

\emph{Case ii). Estimates of $J_3$ with $|\mu|<|\al|$.}

We rewrite the $J_3$ as 
\begin{align*}
	&\sum_{\be+\ga+\mu=\al;|\mu|< |\al|} \int_{\R^2}e^{-q}\d_t\Ga^\al\phi\cdot (\Ga^\be u\cdot\nab(\Ga^\ga u\cdot\nab\Ga^\mu\phi))\ dx\\
	=&\ \sum_{\be+\ga+\mu=\al;|\mu|< |\al|}\int_{\R^2}e^{-q}\d_t\Ga^\al\phi\cdot (\Ga^\be u\cdot\Ga^\ga u\cdot\nab\nab\Ga^\mu\phi))\ dx\\
	&+\sum_{\be+\ga+\mu=\al;|\mu|< |\al|}\int_{\R^2}e^{-q}\d_t\Ga^\al\phi\cdot (\Ga^\be u\cdot \nab\Ga^\ga u\cdot\nab\Ga^\mu\phi))\ dx\\
	=&\ J_{33}+J_{34}\,.
\end{align*}

For the first term $J_{33}$, we can assume that $|\be|\leq |\ga|$. Then by the estimates in \emph{Step 1. Case i) and Case ii)} and \eqref{Decay-u}, we have
\begin{align*}
	J_{33}&\lesssim \|\Ga^{[N/2]} u\|_{\Lf}(\ep\mathscr D_N^2+\ep^3\<t\>^{-1})\les \ep^2\mathscr D_N^2+\ep^4\<t\>^{-3/2+\de}\,.
\end{align*}
For the second term $J_{34}$, when $|\be|,|\mu|\leq [N/2]$, by \eqref{Decay-u} and Lemma \ref{Decay-phi0}, we have
\begin{align*}
	\int_{\R^2}e^{-q}\d_t\Ga^N\phi\cdot (\Ga^{[N/2]} u\cdot \nab\Ga^N u\cdot\nab\Ga^{[N/2]}\phi))\ dx &\les \|\d_t\Ga^N\phi\|_{L^2}\|\Ga^{[N/2]}u\|_{\Lf}\|\nab\Ga^N u\|_{L^2}\|\nab\Ga^{[N/2]}\phi\|_{\Lf}\\
	&\lesssim \ep^3\<t\>^{-1+\de}\mathscr D_N\les \ep^2\mathscr D_N^2+\ep^4\<t\>^{-2+2\de}\,.
\end{align*}
When $|\ga|,|\mu|\leq [N/2]$, by \eqref{decay-du0}, Lemma \ref{Decay-phi0}, \eqref{Decay-u} and \eqref{XX_ka}, we have
\begin{align*}
	&\int_{\R^2}e^{-q}\d_t\Ga^N\phi\cdot (\Ga^{N} u\cdot \nab\Ga^{[N/2]} u\cdot\nab\Ga^{[N/2]}\phi))\ dx\\
	&\lesssim \ep\big(\|\Ga^{N} u\|_{L^2}\|\nab\Ga^{\leq 1} u\|_{\Lf}\|\nab\Ga^{\leq 1}\phi\|_{\Lf}
	+\|\Ga^{N-2}u\|_{\Lf}\|\nab\Ga^{[N/2]} u\|_{L^2}\|\nab\Ga^{[N/2]}\phi\|_{\Lf}\big)\\
	&\lesssim \ep \Big( \ep^2\<t\>^{-1+\de}\ep\<t\>^{-1/2}+\<t\>^{-1/2+\de}(\ep+\ep\<t\>^{1/2}\mathscr D_N)\mathscr D_N\ep\<t\>^{-1/2})\\
	&\les \ep^4 \<t\>^{-3/2+\de}+\ep^3\<t\>^{-1+\de}\mathscr D_N+\ep^3\<t\>^{-1/2+\de}\mathscr D_N^2\\
	&\les \ep^4 \<t\>^{-3/2+\de}+\ep^2\mathscr D_N^2\,.
\end{align*}
When $|\be|,|\ga|\leq [N/2]$, by \eqref{Decay-u}, \eqref{decay-du0}, Lemma \ref{Decay-phi0} and \eqref{XX_ka}, we have
\begin{align*}
	&\int_{\R^2}e^{-q}\d_t\Ga^N\phi\cdot (\Ga^{[N/2]} u\cdot \nab\Ga^{[N/2]} u\cdot\nab\Ga^{N}\phi))\ dx\\
	\lesssim&\  \ep\|\Ga^{\leq 1} u\|_{\Lf}\|\nab\Ga^{\leq 1}u\|_{\Lf}\|\nab\Ga^{|\al|-1}\phi\|_{L^2}\\
	&+\ep\|\Ga^{[N/2]} u\|_{\Lf}\|\nab\Ga^{[N/2]}u\|_{L^2}\|\nab\Ga^{ N-2}\phi\|_{\Lf}\\
	\lesssim&\  \ep^4\<t\>^{-3/2+\de}+\ep^2\<t\>^{-1/2+\de}\mathscr D_N \<t\>^{-1/2}(\ep+\<t\>^{1/2}\mathscr D_N)\\
	\les&\  \ep^4\<t\>^{-3/2+\de}+\ep^2\mathscr D_N^2\,.
\end{align*}
Thus we also obtain
\begin{align*}
	J_{34} \lesssim \ep^4\<t\>^{-3/2+\de}+\ep^2\mathscr D_N^2\,.
\end{align*}

\bigskip 

In conclusion, by the above three steps, the desired high order energy estimate \eqref{E-phi} is obtained. We complete the proof of the proposition.
\end{proof}

\bigskip 
\section{Proof of Proposition \ref{Main_Prop}}\label{sec-proof}
In this section, we prove the bootstrap Proposition \ref{Main_Prop} by the energy estimates in Proposition \ref{Prop_Eesti} and the $L^2$ weighted estimates in Proposition \ref{propXX}. This closes our proof of Theorem \ref{Main_thm}.

\begin{proof}[Proof of of Proposition \ref{Main_Prop}]
	\ 
	
	\emph{Step 1. We prove that for any $j\leq N$}
	\begin{align}   \label{E-uint}
		\EE_j(u,t)+\int_0^t\mathscr D^2_j(u,s)ds\leq C_{3j} \EE_j(u,0)+ C_{3j} \ep^3\ln(e+t)\,.
	\end{align}
	
	We recall the energy estimates \eqref{E-u} in Proposition \ref{Prop_Eesti}
	\begin{align*}%       \label{E-ure}
		\frac{d}{dt}\EE_\al(u,t)+\frac{1}{8}\mathscr D_\al^2(u,t)\leq C_1 \mathscr D_{|\al|-1}^2(u,t)+ C_1\ep^3 \<t\>^{-1}\,.
	\end{align*}
Integrating in time yields
\begin{align}       \label{E-ualint}
	\EE_\al(u,t)+\frac{1}{8}\int_0^t\mathscr D_\al^2(u,s)ds\leq \EE_\al(u,0)+ C_1 \int_0^t\mathscr D_{|\al|-1}^2(u,s)ds+ C_1\ep^3 \ln(e+t)\,.
\end{align}
    This estimate with  $|\al|=0$ implies that the bound \eqref{E-uint} with $j=0$ holds. Then we prove \eqref{E-uint} by induction on $j$. Precisely, we assume that the bound \eqref{E-uint} for $j=n-1$ holds, i.e.
    \begin{align} \label{j-1ass}
    	\EE_{n-1}(u,t)+\int_0^t\mathscr D^2_{n-1}(u,s)ds\leq C_{3,n-1}\EE_{n-1}(u,0)+ C_{3,n-1}\ep^3\<t\>^{-1}\,.
    \end{align}
    Multiplying \eqref{j-1ass} with $2C_1$, and add \eqref{E-ualint} with $|\al|=n-1$, we obtain the bound \eqref{E-uint} for $j=n$. Hence, the bound \eqref{E-uint} is obtained.

\medskip 
    \emph{Step 2. We prove that}
    \begin{align}   \label{E-int}
    	E_N(t) \leq C_4 E_N(0)+C_4\ep^3\ln(e+t)\,.
    \end{align}
    
    We recall the energy estimates \eqref{E-phi}
	\begin{align*}
		\frac{d}{dt}(\EE_\al(\phi,t)+\tilde \EE_\al(\phi,t))+\mathscr D_\al^2(\phi,t)\leq C_2 \ep \mathscr D^2_{N}(t)+C_2\ep^3 \<t\>^{-1}\,.
	\end{align*}
Integration in time and summation over $|\al|\leq N$ yield
\begin{align*}
	\EE_N(\phi,t)+\tilde \EE_N(\phi,t)+\int_0^t \mathscr D_N^2(\phi,s)ds\leq \EE_N(\phi,0)+\tilde \EE_N(\phi,0)+C_2 \ep \int_0^t \mathscr D_N^2(s)ds +C_2 \ep^3 \ln(e+t).
\end{align*}
This bound and \eqref{E-uint} with $j=N$ imply
\begin{align*}
	&\EE_N(u,t)+\EE_N(\phi,t)+\tilde \EE_N(\phi,t)+\int_0^t \big( \mathscr D_N^2(u,s)+ \mathscr D_N^2(\phi,s)\big)\ ds\\
	&\leq \EE_N(u,0)+ \EE_N(\phi,0)+\tilde \EE_N(\phi,0)+C_2 
	\ep \int_0^t \mathscr D_N^2(s)ds +C_2 \ep^3 \ln(e+t)
\end{align*}
Since $\ep$ is sufficiently small, the fourth term in the right hand side can be absorbed by the left hand side. Then \eqref{E-int} is a consequence of \eqref{eq}.

\medskip 
\emph{Step 3. We show the improve bound \eqref{Main_Prop_result1} for energy $E_N$.}

By initial data assumption \eqref{MainAss_dini} and \eqref{E-int}, we have
\begin{align*}
	E_N(t)\leq C_4 \ep^2+C_4\ep^3\ln(e+t)
\end{align*}
Then choosing 
\[C_0\geq 2\sqrt{2C_4}\,,\quad C_T=\frac{C_0^2}{8C_4}\,,\]
on the interval $[0,T]$ for $T\leq e^{C_T/\ep}$, we have
\begin{align*}
	E_N(t)\leq (\frac{C_0}{2}\ep)^2\,.
\end{align*}

We recall the estimate \eqref{XX_ka-2}
\begin{align*}
	\XX_{N-2}^{1/2}\leq C_5 E_{N-2}^{1/2}+C_5 \ep^2\,.
\end{align*}
Then choosing 
\begin{align*}
	C_0\geq \max\{2\sqrt{2C_4},4C_5\}
\end{align*}
and $\ep$ sufficiently small, we have
\begin{align*}
	\XX^{1/2}_{N-2}\leq \frac{C_0^2}{4}\ep+\frac{C_0}{4}\ep^2 \leq \frac{C_0^2}{2}\ep\,.
\end{align*}

Thus we obtain the bound \eqref{Main_Prop_result1}. This completes the proof of Proposition \ref{Main_Prop}.
\end{proof}

\bigskip 
\section*{Acknowledgment}
J. Huang was supported by Beijing Institute of Technology Research Fund Program for Young Scholars, as well as the NSFC Grant No. 12271497. The author N. Jiang is supported by the grants from the National Natural Foundation of China under contract Nos. 11971360 and 11731008, and also supported by the Strategic Priority Research Program of Chinese Academy of Sciences, Grant No. XDA25010404. L. Zhao is supported by NSFC Grant of China No. 12271497 and the National Key Research and Development Program of China No. 2020YFA0713100.

\end{document}